\documentclass[12pt]{amsart}
\usepackage{amssymb}
\usepackage{amsfonts}
\usepackage{amscd}
\usepackage{latexsym}
\usepackage{mathrsfs}
\usepackage{xy} \xyoption{all}

\newcounter{cs}
\stepcounter{cs}
\newcounter{ds}
\stepcounter{ds}

\newcommand{\casos}{\begin{itemize}}
\newcommand{\fcasos}{\end{itemize}\setcounter{cs}{1}}

\newcommand{\Cs}{{$C^*$-algebra}}
\newcommand{\cK}{\mathcal{K}}
\newcommand{\ep}{\varepsilon}
\newcommand{\R}{\mathbb{R}}
\newcommand{\N}{\mathbb{N}}
\newcommand{\Z}{\mathbb{Z}}
\newcommand{\diag}{{\mathrm{diag}}}
\newcommand{\Cu}{{\mathrm{Cu}}}

\newtheorem{lem}{Lemma}[section]
\newtheorem{corol}[lem]{Corollary}
\newtheorem{theor}[lem]{Theorem}
\newtheorem{prop}[lem]{Proposition}

\theoremstyle{definition}
\newtheorem{defi}[lem]{Definition}

\newtheorem{exem}[lem]{Example}

\newtheorem{rema}[lem]{Remark}

\begin{document}

\title[The Corona Factorization property and stability]{The Corona
  Factorization property, Stability, and the Cuntz semigroup of a $C^*$-algebra}%
\author{Eduard Ortega}
\address{Department of Mathematics and Computer Science, University of Southern Denmark, Campusvej 55, DK-5230, Odense M, Denmark} \email{ortega@imada.sdu.dk}
\author{Francesc Perera}
\address{Departament de Matem\`atiques, Universitat Aut\`onoma de Barcelona, 08193 Bellaterra, Barcelona, Spain} \email{perera@mat.uab.cat}
\author{Mikael R\o rdam}
\address{Department of Mathematical Sciences, University of Copenhagen, Universitets\-parken 5, DK-2100, Copenhagen \O, Denmark} \email{rordam@math.ku.dk}

\thanks{} 
\date{\today}

\begin{abstract}
The Corona Factorization Property, originally invented to study
extensions of \Cs s, conveys essential information about the
intrinsic structure of the \Cs. We show that the Corona Factorization
Property of a
$\sigma$-unital $C^*$-algebra is completely captured by its Cuntz
semigroup (of equivalence classes of positive elements in the
stabilization of $A$). The corresponding condition in the Cuntz semigroup
is a very weak
comparability property termed the Corona Factorization
Property for semigroups. Using this result one can for example
show that all unital $C^*$-algebras with finite decomposition rank
have the Corona Factorization Property.

Applying similar techniques we study the related question of when \Cs s
are stable. We give an intrinsic characterization, that we term
property (S), of \Cs s that have no non-zero unital quotients and
no non-zero bounded 2-quasitraces. We then show that property (S) is
equivalent to stability provided that the Cuntz semigroup of the \Cs{}
satisfies another (also very weak) comparability property, that we
call the $\omega$-comparison property.
\end{abstract}
\maketitle

\section{Introduction}
\noindent The Corona Factorization Property was defined and studied
by Kucerovsky and Ng in \cite{KN} building on work by
Elliott and Kucerovsky, \cite{EK}, in which \emph{purely large} \Cs
s were studied. Both concepts relate to the theory of extensions and
in particular to the important question on when extensions
are automatically absorbing.

A \Cs{} satisfies the Corona Factorization Property if every full
projection in the multiplier algebra of its stabilization is
properly infinite (and hence equivalent to the unit). The existence
of non-properly infinite full projections in the multiplier algebra
of a stable \Cs{} was noted (implicitly) in \cite{Ro2}, and more
explicitly in \cite{Ro4}, in connection with the construction of
non-stable \Cs s  that become stable when being tensored with a
matrix algebra. The existence of finite full projections in the
multiplier algebra of a stable \Cs{} was also essential in the
construction in \cite{Ro3} of a simple \Cs{} with a finite and an
infinite projection. In the language of Kucerovsky and Ng it is
shown in \cite{Ro3} that the \Cs{} $C(\prod_{n=1}^\infty S^2)$ does
not have the Corona Factorization Property.

Zhang proved a (partial) converse of these results, that a simple
\Cs{} of real rank zero with the Corona Factorization Property is
either stably finite or purely infinite.

It thus appears that failure to have the Corona Factorization
Property is an ``infinite dimensional'' property, and conversely
that all \Cs s with ``finite dimensional behavior'' should have the
Corona Factorization Property. (By finite and infinite dimensinality
we are, of course, not referring to the vector space dimension of
the \Cs{}, but rather to its non-commutative dimension---perhaps
best defined through Kirchberg and Winter's notion of decomposition
rank.) Pimsner, Popa, and Voiculescu studied in \cite{PPV} extensions
of $C(X)\otimes \mathcal{K}$, where $X$ is a finite-dimensional
compact metric space, and developed an $\mathrm{Ext}(X,-)$ theory.
It follows in particular from their work that $C(X) \otimes \cK$ has
the Corona Factorization Property when $X$ has finite dimension. The
assumption that $X$ is finite
dimensional is crucial.

Using Kirchberg and Winter notion of decomposition rank, \cite{KW},
mentioned above, Kucerovsky and Ng, \cite{KN2}, studied extensions
of type I $C^*$-algebras with finite decomposition.

In this paper we show that a $\sigma$-unital $C^*$-algebra (simple or
not) satisfies the Corona Factorization Property if, and only if, its
Cuntz semigroup satisfies a very weak comparison property that we
call the Corona Factorization Property for semigroups (also considered
in \cite{OPR}).
We also introduce stronger notions of comparison for ordered abelian
semigroups, some of which
are verified for the Cuntz semigroup of a unital \Cs{} with
finite decomposition rank,
and so the Corona Factorization Property also holds for these
algebras.

This parallels the property that the
authors introduced and examined in the article \cite{OPR}. There it
was shown, using entirely different techniques than those employed here,
that a $\sigma$-unital \Cs{} of \emph{real rank
zero} has the Corona Factorization Property if and only if its
monoid $V(A)$ of Murray-von Neumann equivalence
classes of projections in the stabilization of a \Cs{} $A$ has the
Corona Factorization Property (for monoids). 

In outline the paper is as follows. In Section~\ref{sec:semigroups},
we define and consider various comparability properties for ordered abelian
semigroups, including \emph{$n$-comparison} and
\emph{$\omega$-comparison} (and their weak counterparts), and the
\emph{Corona Factorization Property} for semigroups. These properties can be
viewed as weakened forms of the almost unperforation
property for semigroups (considered in \cite{Ro1}). In fact, an
ordered abelian semigroup has the $0$-comparison property if and only if it
is almost unperforated. It follows from a result of Toms and Winter,
\cite{TW}, that the Cuntz semigroup of any unital simple \Cs{} with
decomposition rank $n$ has the $n$-comparison property (and hence also
$\omega$-comparison and the Corona Factorization Property). It was this
result by Toms and Winter that led us to consider
$n$-comparison. 

In Section~\ref{sec:CFP} we establish (using the above mentioned
result of Toms and Winter) that the Cuntz semigroup of (non-simple)
unital \Cs s with finite decomposition rank has the weak
$n$-comparison property, and hence also the Corona Factorization
property. After the first version of this
paper was posted  Leonel Robert proved that the Cuntz semigroup of any
\Cs{} with nuclear dimension at most $n$ (respectively, $\omega$) has
$n$-comparison (respectively, $\omega$-comparison), \cite{Robert}. 

In Section~\ref{sec:stable} we consider an intrinsic property (that we call
property (S)) of a \Cs{} defined in terms of Cuntz'
comparison theory for \Cs s, and we show that it is equivalent to
the absence of unital quotients and bounded 2-quasitraces. It is
further shown that property (S) is equivalent to stability of a
\Cs{} if its Cuntz semigroup has the $\omega$-comparison property,
thus generalizing a result from \cite{HRW}.

In Section~\ref{sec:CFP_Cuntz} we prove our main result,
that a $\sigma$-unital \Cs{} has the Corona Factorization Property if
and only if its Cuntz semigroup satisfies the comparability condition
called the Corona Factorization Property
(for semigroups), and that every ideal in a $\sigma$-unital
\Cs{} has the Corona Factorization Property if and only if its Cuntz
semigroup satisfies  the strong Corona Factorization Property (for
semigroups).

\section{Comparability in ordered abelian semigroups} \label{sec:semigroups}

\noindent In this section we shall discuss a number of comparability
properties of ordered abelian semigroups.

Consider an  ordered abelian semigroup $(W,+,\leq)$ that we here and
in what is to follow tacitly
shall assume to be \emph{positive}, i.e., $x \le x+z$ for all $x,z \in
W$. We do 
not assume that the order is the algebraic one, given by $x \le y$
if and \emph{only if} $y = x+z$ for some $z$ in $W$.

A \textit{state on} $W$
\textit{normalized at} $x\in W$ is an additive order preserving map
from $W$ into $\mathbb{R}^+\cup\{\infty\}$ that maps $x$ to $1$. The
set of all states normalized at $x$ is denoted by $S(W,x)$.
Given two elements $x,y\in W$, one writes $x\propto y$ if there
exists $n\in \mathbb{N}$ such that $x\leq n y$.

The result below has appeared already in several versions in the
literature, perhaps first as the extension result of Goodearl and
Handelman in \cite[Lemma 4.1]{GH:Extending}. We wish to emphasize
the following formulation that will be essential for our paper.

\begin{prop}\label{stable_comparison} Let $(W,+,\leq)$ be an ordered
  abelian semigroup, and
let $x,y\in W$. Then the following statements are equivalent:
\begin{enumerate}
\item There exists $k\in \mathbb{N}$ such that $(k+1)x\leq ky$.
\item There exists $k_0\in \mathbb{N}$ such that $(k+1)x\leq k y$
for every $k\geq k_0$.
\item $x\propto y$ and $f(x)<f(y)$  for every $f\in S(W,y)$.
\end{enumerate}
\end{prop}

\noindent
If $S(W,y)$ is empty, then statement (iii) above reduces to the
statement $x \propto y$. 

\begin{proof} (iii) $\Rightarrow$ (i). Assume first that $y$ is an
  order unit for $W$ and that (iii) holds. Use compactness of $S(W,y)$
  to find natural numbers $m >
  m'$ such that $f(mx) < f(m'y)$ for all $f \in S(W,y)$. By the
  results of Goodearl and
Handelman in \cite{GH:Extending} there exist a natural number $n$ and
an element $z$ in $W$ such that $nmx + z \le nm'y + z$. From this
inequality we get $rnmx + z \le rnm'y + z$ for all natural numbers
$r$. Now, $z \le py$ for some natural number $p$ (because $y$ was
assumed to be an order unit). This gives us $rnmx \le (rnm'+p)y$ for
all natural numbers $r$. Choose $r$ large enough so that $rnm >
rnm'+p$ and put $k=rnm'+p$. Then $(k+1)x \le rnmx \le ky$. 

If $y$ is not an order unit, then consider the order ideal $W'$ in $W$
generated by $y$ (i.e., $W'$ consists of all elements $w$ such that $w
\propto y$). Then $y$ is an order unit in $W'$ and $x$ belongs to
$W'$. Moreover, every state in $S(W,y)$ restricts to a state in
$S(W',y)$; and every state in $S(W',y)$ extends to a state in $S(W,y)$
by letting it attain the value $\infty$ in $W \setminus W'$. In this
way we can reduce the general case to the case dealt with above where 
$y$ is an order unit.

(ii) $\Rightarrow$ (iii). This is contained in the proof of
\cite[Proposition 3.2]{Ro1}, but here it comes again:
 First, if $(k+1)x \le ky$, then $x
\le ky$, so $x \propto y$. Second,  $(k+1)x \le ky$ implies that
$f(x) \le k(k+1)^{-1} < 1 = f(y)$ for every $f\in S(W,y)$.

(i) $\Rightarrow$ (ii). Suppose $(m+1)x\leq my$ for some positive integer
$m$. Put $k_0= (m+1)m$. For each $k \ge k_0$  write $k = (m+1)r
+ s$, where $r$ and $s$ are non-negative integers with $s \le
m$. Note that necessarily $r \ge m$. Therefore
$$(k+1)x=(m+1)rx+(s+1)x \leq (m+1)rx+(m+1)x \le mry + my
\le ky.$$
\end{proof}

\begin{defi} Given $x,y$ in an ordered abelian semigroup $W$. Then we
  say that $x$ is \emph{stably dominated} by $y$, written $x<_s y$, if
  the equivalent
conditions (i)--(iii) in Proposition~\ref{stable_comparison} hold.
\end{defi}

\begin{rema}
Notice that the relation $<_s$ is transitive. Indeed, let $x,y,z\in
W$ be such that  $x<_s y$ and $y<_s z$. Then by
Proposition~\ref{stable_comparison}~(ii) there exists
$k\in\mathbb{N}$ such that $(k+1)x\leq ky$ and $(k+1)y\leq kz$. But
then $(k+1)x\leq ky\leq (k+1)y\leq kz$, so $x<_sz$.
\end{rema}

\begin{rema} \label{rem:almunp}
The notion of almost unperforation from \cite[Definition 3.3]{Ro1}
can, in terms of stable domination, be rephrased as
  follows. An ordered abelian semigroup $W$ is almost unperforated if
  and only if for all $x,y$ in $W$, $x <_s y$ implies $x \le
  y$.
\end{rema}

\noindent Let us briefly remind the reader about the ordered Cuntz
semigroup(s) $W(A)$ and $\Cu(A)$ associated to a $C^*$-algebra $A$
introduced by Cuntz in \cite{Cun}. The (uncompleted) Cuntz semigroup
$W(A)$ arises as $M_\infty(A)^+/\!   \approx$, where 
$M_{\infty}(A)^+$ denote the positive elements in
the disjoint union $\bigcup_{n=1}^\infty
M_n(A)$, and where $\approx$ is Cuntz equivalence: For $a\in
M_n(A)^+$ and $b\in M_m(A)^+$ write $a\precsim b$ if $x_k^*bx_k
\to a$ for some sequence $\{x_k\}$ in $M_{m,n}(A)$, and write $a
\approx b$ if $a \precsim b$ and $b \precsim a$. Equip $W(A)$ with
addition arising from direct sum $a\oplus b=
\diag(a,b)\in M_{n+m}(A)$, and with
the order induced by $\precsim$. Let
$\langle a \rangle \in W(A)$ denote the equivalence class containing
$a$. 

In \cite{CEI}, Coward, Elliott, and Ivanescu
introduced a modified version of the Cuntz semigroup, $\Cu(A)$. They
use suitable equivalence classes of countably generated Hilbert
modules (which, in the case of stable rank one, amount to
isomorphism) to obtain an ordered abelian semigroup, which in certain
good situations can be
viewed as the completion of $W(A)$, see \cite{ABP}, 
and which turns out to be order-isomorphic
to $W(A\otimes\cK)$. Unlike the functor $A \mapsto W(A)$, the new
Cuntz semigroup functor $A \mapsto \Cu(A)$ is continuous; and
$\Cu(A)$ is closed under suprema of increasing sequences. 

Coward, Elliott, and Ivanescu also formalize the notion of
compact containment in the Cuntz semigroup to arbitrary ordered
abelian semigroups $W$ that admit suprema of increasing sequences as
follows: Given two elements $x,y$ in $W$, $x$ is \emph{compactly
  contained} in $y$, sometimes colloquially referred to as
$x$ is \emph{way below} $y$, and
denoted by $x\ll y$, if whenever $\{y_n\}$ is an
increasing sequence with supremum greater than or equal to $y$,
eventually $x\leq y_n$. 

A proto-example of compact containment in the Cuntz semigroup is
$\langle (a-\ep)_+ \rangle \ll \langle a \rangle$ that holds in
$\Cu(A)$ for all \Cs s $A$, for all positive elements $a \in A
\otimes \cK$, and for all $\ep >0$. 

\begin{defi}[Complete ordered abelian semigroups] \label{def:complete} 
An ordered abelian semigroup $(W,+,\le)$ is said to be \emph{complete}
if every increasing sequence in $W$ admits a supremum; and if for
every pair of elements $x,y \in W$ the following holds: $x \le y$ if
and only if $x' \le y$ for all $x' \in W$ with $x' \ll x$.
\end{defi}

The Cuntz semigroup $\Cu(A)$ is complete and in fact belongs to the
category {\bf{Cu}}, defined in \cite{CEI}, which is the appropriate frame
in which to study this object. 
For our purposes the weaker notion of completeness
will suffice. 

Identifying $A$ and matrix algebras over $A$ as hereditary
subalgebras of $A \otimes \cK$ we can regard $W(A)$ as being a
subsemigroup of $\Cu(A)$. Each positive element $a$ in $A$ or in a
matrix algebra over $A$ defines a class $\langle a \rangle$ in
$\Cu(A)$.  

Given $a,b\in A^+$ write $a \otimes 1_n$ to denote the $n$-fold direct
sum $a \oplus \cdots \oplus a$. Clearly, $\langle a
\otimes 1_n \rangle = n \langle a \rangle$ in $\Cu(A)$. 
Write $a \prec_s b$ if
$\langle a \rangle <_s \langle b \rangle$. 
About this relation we have the lemma below, which is similar
to \cite[Proposition~2.4]{Ro:UHFII}.

\begin{lem} \label{ep-delta}
Let $a$ and $b$ be positive elements in a \Cs{} $A$ and suppose that
$a \prec_s b$. Then, for each $\ep > 0$ there exists $\delta > 0$
such that $(a-\ep)_+ \prec_s (b-\delta)_+$.
\end{lem}

\begin{proof} For $c \in A^+$ note that $((c
  \otimes 1_k)-\eta)_+ = (c-\eta)_+ \otimes 1_k$.

There exists a positive integer $k$ such that $(k+1)
  \langle a \rangle \le k \langle b \rangle$.
Hence $a \otimes 1_{k+1} \precsim b \otimes 1_k$. Let $\ep >0$. It
then follows from 
\cite[Proposition~2.4]{Ro:UHFII} that there exists $\delta >0$ such
that
$$(a-\delta)_+ \otimes 1_{k+1} = (a \otimes 1_{k+1} - \delta)_+ \precsim
(b \otimes 1_k - \ep)_+ = (b-\ep)_+ \otimes 1_k.$$ This shows that
$(a-\delta)_+ \prec_s (b-\ep)_+$.
\end{proof}

Let us note some properties that can
be deduced from \cite[Proposition 2.4]{Ro:UHFII} and
Lemma~\ref{ep-delta}. If $x,y \in \Cu(A)$, then
\begin{itemize}
\item[(a)] $x \le y$ if and only if $x_0 \le y$ for every $x_0 \in \Cu(A)$
  with $x_0 \ll x$;
\item[(b)]  if $x \le y$ and if $x_0 \in \Cu(A)$ is such that $x_0 \ll x$, then
  there is $y_0 \in \Cu(A)$ with $y_0 \ll y$ and $x_0 \le y_0$;
\item[(c)] if $x <_s y$ and if $x_0 \in \Cu(A)$ is such that $x_0 \ll x$, then
  there is $y_0 \in \Cu(A)$ with $y_0 \ll y$ and $x_0 <_s y_0$.
\end{itemize}

\noindent Our first comparability property, given in
Definition~\ref{def:n-comp} below, is prompted by
a result of Toms and Winter, \cite[Lemma 6.1]{TW}. Recall that if
$\tau$ is a
positive trace (or a 2-quasitrace) in $A$, then one can associate to it a
dimension function $d_\tau \colon \Cu(A) \to [0,\infty]$ given by
\begin{equation} \label{eq:d_tau}
d_\tau(\langle a \rangle) = \lim_{n \to \infty} \bar{\tau}(a^{1/n}),
\end{equation}
where $a$ is a positive element in $A\otimes \cK$ and $\bar{\tau}$ is
the usual extension of $\tau$ to $A\otimes \cK$. The trace property
ensures that $d_\tau$ is 
well-defined. We can also view $d_\tau$ as being a function on the
positive elements in $A\otimes \cK$. We shall
not distinguish between the two situations.

\begin{prop}[Toms and Winter] \label{prop:tw}
Let $A$ be a simple, separable,
  and unital \Cs{} of decomposition rank $n < \infty$. Let $a,d_0,d_1,
  \dots, d_n$ be positive elements in $A$ such that $d_\tau( a ) <
  d_\tau( d_j )$ for $j=0,1, \dots, n$ and for all tracial states
  $\tau$ on $A$ (where $d_\tau$ is the dimension function on $A$
  associated with the trace $\tau$). It follows that
$$\langle a \rangle \le \langle d_0 \rangle + \langle d_1 \rangle +
\cdots + \langle d_n \rangle,$$ in the Cuntz semigroup $\Cu(A)$.
\end{prop}

\begin{defi}[The $n$-comparison property] \label{def:n-comp}
Let $(W,+,\leq)$ be an ordered abelian semigroup and let
  $n$ be a natural number. Then $W$ is said to have the
\emph{$n$-comparison property} if whenever $x,y_0,\ldots,y_n$ are
elements in $W$  with $x<_s y_j$ for all $j$, then $x\leq
y_0+y_1+\cdots+y_n$.
\end{defi}

\noindent Note that $W$ has the $0$-comparison property if and only if $W$ is
almost unperforated, cf.\ Remark~\ref{rem:almunp}. Note also, that if
$W$ has the $n$-comparison property for some $n$, then it has the
$m$-comparison property for all $m \ge n$.

% \begin{rema}
% Recall from \cite{HR} that $F(A)$ is the set of \emph{compactly
% supported} positive elements, that is,
% \[
% F(A)=\{a\in A^+\mid \text{ there exists }e\in A^+\text{ with
% }ea=a\}\,.
% \]

% If $a$ belongs to $F(A \otimes \cK)$, then $a$ is equivalent (in
% the sense of Cuntz comparison) to an element in $M_\infty(A)^+$.
% \end{rema}

With Definition~\ref{def:n-comp} at hand we can rephrase the
proposition of Toms and Winter above as follows:

\begin{prop}\label{prop-tw}
Let $A$ be a simple, separable and unital $C^*$-algebra with
decomposition rank $n < \infty$. Then $W(A)$ and $\Cu(A)$ have the 
$n$-comparison property.
\end{prop}

\begin{proof}
By Lemma~\ref{n-comp:W-Cu} below it suffices to show this for $W(A)$. 

Let $x,y_0,\ldots,y_n\in W(A)$ be  such that $x<_s y_j$ for every
$j=0,\ldots,n$. Upon replacing $A$ by a matrix algebra over $A$
(which does not change the decomposition rank) we may assume that
there are positive elements $a,d_0,d_1, \dots, d_n$ in $A$ such that
$x = \langle a \rangle$ and $y_j = \langle d_j \rangle$.

We know from Proposition \ref{stable_comparison} that $f(x)<f(y_j)$
for every dimension function $f$ on $A$ normalized at $y_j$. As $A$
is simple and unital, every such $f$ is a multiple of a dimension
function which is normalized at the unit: $\langle 1_A \rangle$. It
follows that $d_\tau(a) < d_\tau(d_j)$ for every tracial state
$\tau$ on $A$ (because $d_\tau$ then is a dimension function on $A$
normalized at $1_A$) Thus, by \cite[Lemma 6.1]{TW}  (which in fact is
Proposition~\ref{prop:tw}), we get that $a \precsim d_0 \oplus d_1 
\oplus \cdots \oplus d_n$, which in turn implies that $x\leq
y_0+y_1+\cdots + y_n$, as desired.
\end{proof}

\noindent As mentioned in the introduction, 
Leonel Robert improved this result after the first version
of this paper was made public, see \cite{Robert}. 

\begin{lem} \label{n-comp:W-Cu}
Let $A$ be a \Cs, and suppose that $W(A)$ has the $n$-comparison
property. Then $\Cu(A)$ has the $n$-comparison property. 
\end{lem}

\begin{proof} 
Let $x,y_0,y_1, \dots, y_n$ in $\Cu(A)$ be given
  satisfying $x <_s y_j$ for all $j$. As $x$ is the supremum of an increasing
  sequence of elements $\{x_k\}$ from $W(A)$ each satisfying $x_k \ll
  x$ it suffices to show that $x' \le y_0 + \cdots + y_n$ for every
  $x' \in W(A)$ with $x' \ll x$. Each $y_j$ is likewise the supremum
  of an increasing sequence $\{y_j^{(k)}\}$ of elements from
  $W(A)$. Use property (c) above to conclude that there is a natural
  number $k$ such that $x' <_s y_j^{(k)}$ holds for all $j$. As the
  inclusion mapping $W(A) \hookrightarrow \Cu(A)$ is an order embedding, the
  inequality $x' <_s y_j^{(k)}$ also holds in $W(A)$. It follows that
  $x' \le y_0^{(k)} + \cdots + y_n^{(k)} \le y_0 + \cdots + y_n$ as
  desired.
\end{proof}

We proceed to define a comparison property which is weaker than the
$n$-comparison property for all $n$. 

\begin{defi}[The $\omega$-comparison property] \label{def:wcomp}
A complete ordered abelian semigroup
  $(W,+,\le)$ is said to have the \emph{$\omega$-com\-pa\-ri\-son property} if
  whenever $x',x,y_0,y_1,y_2, \dots$ are elements in $W$ such that $x <_s
  y_j$ for all $j$ and $x'\ll x$, then $x' \le y_0+y_1+
  \cdots + y_n$ for some $n$
  (that may depend on the element $x'$).
\end{defi}

\noindent It is clear that if $W$ has the $n$-comparison property for
some $n$, then $W$ also has the $\omega$-comparison property. 

We shall now (re-)define two even weaker comparison properties for an
ordered abelian semigroup, the \emph{strong Corona Factorization
  Property}  and the \emph{Corona Factorization
  Property}. These were also defined (slightly differently)
in our paper, \cite{OPR}, written in
parallel with this paper. In \cite{OPR} we were interested in the
case of the (discrete) algebraically ordered semigroup, $V(A)$, of
Murray-von Neumman equivalence classes of projections. In the present
case of the Cuntz semigroup, one has to take into consideration its
continuous aspect where
the definition below is more appropriate (and it extends the definition
given in \cite{OPR}, see Remark \ref{rem:int}). We shall sometimes
refer to the definitions of the Corona Factorization Property in
\cite{OPR} as the discrete versions and the ones given below as the
continuous verions.

\begin{defi}[The strong Corona Factorization Property for semigroups]
Let \\ $(W,+,\leq)$ be a complete ordered abelian semigroup. Then
  $W$ is said to satisfy the \emph{strong Corona Factorization Property}
  if given any $x',x$ in $W$, any sequence
$\{y_n\}$ in $W$, and any positive integer $m$
satisfying $x'\ll x$ and $x\leq my_n$ for all $n$, then there
exists a positive integer $k$ such that $x'\leq y_1+y_2 +\cdots+y_k$.
\end{defi}

\noindent Fullness, as defined below, was also considered
in \cite{OPR}, and again we must extend this definition so that it
applies to complete ordered (positive) semigroups.

\begin{defi}[Full elements and sequences]
Let $(W,+,\leq)$ be a complete ordered abelian semigroup.

An element $x$ in $W$
is said to be \emph{full} if for any $y',y \in
W$ with $y'\ll y$, one has $y'\propto  x$.

A sequence $\{x_n\}$ in $W$ is said to be \emph{full}
if it is increasing and for any $y',y \in
W$ with $y'\ll y$, one has $y'\propto  x_n$ for some (hence all
sufficiently large) $n$.
\end{defi}

\noindent Every order unit in $W$ is full, but the reverse is not
always true. The constant sequence $\{x_n\}$, with $x_n = x$ for all $n$,
is full if and only if $x$ is full.

Suppose that $A$ is a \Cs{} and that $\{a_n\}$ is a sequence of
positive elements in $A$. Then $\{\langle a_n \rangle \}$ is full
in $\Cu(A)$ if and only if $a_1 \precsim a_2 \precsim a_3 \precsim
\cdots$ and $\{a_n\}$ is full in $A$ (in the sense of $\{a_n\}$
not being contained in a proper closed two-sided ideal in $A$).

\begin{defi}[Corona Factorization Property for semigroups]
Let $(W,+,\leq)$ be a complete ordered abelian semigroup. Then
  $W$ is said to satisfy the \emph{Corona Factorization Property}
  if given any
full sequence $\{x_n\}$ in $W$, any sequence
$\{y_n\}$ in $W$, an element $x'$ in $W$, and a positive integer $m$
satisfying $x'\ll x_1$ and $x_n\leq my_n$ for all $n$, then there
exists a positive integer $k$ such that $x'\leq y_1+y_2 +\cdots+y_k$.
\end{defi}

\begin{rema} The content of this remark was suggested to us by
  Leonel Roberts and George Elliott. 

Suppose that $W$ is a complete
  ordered abelian semigroup that contains a full sequence (this is the
  case whenever $W = \Cu(A)$ where $A$ is a separable \Cs). Then,
  letting $u$ be the supremum of any full sequence in $W$, it follows
  that the supremum, $p$, of the sequence $\{nu\}$ is the largest
  element in $W$, i.e., $w \le p$ for all $w \in W$. In
  particular, $p$ is properly
  infinite because $2p \le p$. Obviously $p$ is unique with these properties.

Consider the following property (Q) of $W$: For every $u \in W$ and for
every natural number $m$, if $mu = p$, then $u=p$. (Equivalently, one
can ask that $u$ be properly infinite if a multiple of $u$ is equal to
$p$.) We show below that
property (Q) implies the Corona Factorization Property, but we do not
know if the two properties are equivalent in general. One could
consider the following stronger property (QQ) of an ordered abelian
semigroup $W$: every element $u$ in $W$ a multiple of which is
properly infinite is itself properly infinite. Then clearly (QQ)
implies (Q). It is shown in \cite[Theorem 5.14]{OPR} 
that any refinement monoid that satisfies the Corona Factorization
Property also satisfies (Q), so the two conditions are equivalent in
this case. It is also shown in \cite[Theorem 5.14]{OPR} 
that any refinement monoid that satisfies the strong Corona Factorization
Property will satisfy (QQ) (and hence (Q)).
In \cite{BoRo} it is shown that any complete ordered abelian
semigroup with $\omega$-comparison satisfies (QQ) (and hence (Q)).

We proceed to show that (Q) implies the Corona Factorization Property.
Indeed, let $\{x_n\}$ be a full sequence in $W$, let $x' \in W$ be
such that $x' \ll x_1$, let $\{y_n\}$ be a sequence in $W$ such that $x_n
\le my_n$ for all $n$ and for some fixed $m \in \N$. Assume that
condition (Q) above holds. We must show that $x' \le y_1 + \cdots + y_k$ for
some $k$. 

Note first that the supremum of the sequence $\{x_1+\cdots + x_n\}$ is
$p$. For all $w' \ll w$ in $W$ we have $w' \le \ell x_n$ for some
natural numbers $\ell, n$. 
Hence $w' \le x_1 + \cdots + x_{\ell+n-1}$ which again is less
than the supremum of the sequence  $\{x_1+\cdots + x_n\}$.   
Let $u$ be the supremum of the sequence $\{y_1+\cdots+y_n\}$. Then $mu
\ge m(y_1+ \cdots + y_n) \ge x_1+ \cdots + x_n$ for all $n$. Therefore
$mu \ge p$, whence $mu=p$. By assumption, $u=p$. In particular we have
that $x_1 \le u$. By the definition of compact containment (and of
$u$) we obtain that $x' \le y_1 + \cdots + y_k$ for some $k$.
\end{rema}

\noindent It is clear that any semigroup that satisfies the strong
Corona Factorization Property also satisfies the Corona Factorization
Property. It was shown in \cite{OPR} that a conical refinement monoid
satisfies the strong Corona Factorization Property if and only if
every ideal of the monoid satisfies the Corona Factorization
Property. It is not clear if this remains true without assuming the
refinement property, but we shall show
(implicit in Theorem~\ref{sCFP-C*-alg}) that this
also holds for semigroups arising as the Cuntz semigroup of a
$\sigma$-unital \Cs. We shall also see, as in the case of \Cs s
with real rank zero, that the Corona Factorization Property defined for
semigroups matches the corresponding property for the \Cs s (see
Section 5).

If $W$ is an algebraically ordered
semigroup, then
the definitions given above for the (strong) Corona Factorization
Property agree with the corresponding ones in \cite{OPR}. The
connection between the two notions is discussed in the remark below.

\begin{rema}\label{rem:int}
{\rm Recall that an \emph{interval} in an ordered abelian semigroup $V$
  is a non-empty subset $I$ of $V$ which is order-hereditary and
  upwards directed. An interval $I$ is said to be \emph{countably
    generated} if there is a sequence $\{x_n\}$ of elements in $I$
  (that can be taken to be increasing) such that $I=\{x\in V\mid x\leq
  x_n\text{ for some }n\}$. 

The addition of two intervals $I$ and $J$ in $V$ is defined to
be 
$$I+J=\{x\in V\mid x\leq y+z\text{ with }y\in I, z\in J\}.$$ 
Denote
by $\Lambda(V)$ the ordered abelian semigroup of all intervals in $V$,
where the order is given by set inclusion, and denote by
$\Lambda_{\sigma}(V)$ the 
subsemigroup of $\Lambda(V)$ whose elements are the countably
generated 
intervals in $V$. Both semigroups are
complete, cf.\ Definition~\ref{def:complete}: 
The supremum of an increasing sequence $\{I_n\}$ of
intervals is $\cup_{n=1}^\infty I_n$. Morever, if $I$ is an interval,
then $[0,x] \ll I$ for every $x \in I$. It follows immediately from this
observation that if $I$ and $J$ are intervals such that $I' \subseteq
J$ for all $I' \ll I$, then $I \subseteq J$.

Now assume that $V$ is algebraically ordered. Then $V$ satisfies
the (strong) Corona Factorization Property (the discrete version
defined in \cite{OPR}) 
if, and only if,
$\Lambda_{\sigma}(V)$ satisfies the
(strong) Corona Factorization Property (the continuous version defined here). 
Let us sketch part of the
arguments needed. Assume that $V$ has the strong Corona Factorization
Property, let $X$, $Y_1,Y_2,\dots$ be elements in
$\Lambda_{\sigma}(V)$, and let
$m\in \N$ be such that $X\subseteq mY_n$ for all $n$. Let $X'\ll
X$ be given. Since $X$ is countably generated by a sequence, say
$\{x_n\}$, that without loss of generality can be assumed to be
increasing, there is $i$ such that $X'\subseteq [0,x_i]$.
It suffices to check that $[0,x_i]\subseteq Y_1+\cdots+Y_l$ for some
$l$. As $x_i$ belongs to $mY_n$ for all $n$, and since each $Y_n$ is
upwards directed, there is $y_n \in Y_n$ such that $x_i \le m y_n$. By
the assumption that $V$ satisfies the strong Corona Factorization
Property, this implies that $x_i\leq
y_1+y_2+ \cdots +y_n$ for some $n$. This again implies
$X'\subseteq Y_1+\cdots + Y_n$, and thus proves that
$\Lambda_{\sigma}(V)$ has the strong
Corona Factorization Property.

In \cite{OPR}, the authors proved that a \Cs{} with real rank zero
has the (strong) Corona Factorization Property if and only if the
projection semigroup $V(A)$ has the corresponding property (discrete
version) for semigroups. Since $V(A)$ is algebraically ordered, it follows that
$\Lambda_{\sigma}(V(A))$ satisfies the (strong) Corona
Factorization Property (continuous version) 
if and only if $V(A)$ satisfies the corresponding property (the
discrete version). In particular,
since for a \Cs{} with real rank zero $A$, one has that $\Cu(A)$ is
order-isomorphic to $\Lambda_{\sigma}(V(A))$---see \cite[Theorem 5.7]{ABP}
---,
our observations yield that, within this class, $V(A)$ has the
(strong) Corona Factorization Property (discrete version) 
if and only if $\Cu(A)$ has the corresponding property (continuous
version), if and only if $A$ has the (strong) Corona Factorization
Property for \Cs s---the latter equivalence follows from the results in
\cite{OPR}.}
\end{rema}

\begin{prop}\label{omega-CFP}
Any complete abelian ordered semigroup, which satisfies the
$\omega$-comparison property, also satisfies the Corona
Factorization Property.
\end{prop}

\begin{proof} Let $W$ be a complete abelian ordered semigroup with the
  $\omega$-comparison property. Let $\{x_n\}$ be a full sequence in
  $W$, let $\{y_n\}$ be another sequence in $W$, let $x' \in W$, and
  let $m$ be a positive integer such that $x' \ll x_1$ and $x_n \le m
  y_n$ for all $n$. For each integer $n \ge 0$ put
$$z_n = y_{n(m+1)+1}+y_{n(m+1)+2}+\cdots+y_{n(m+1)+m+1}.$$
Then
\begin{eqnarray*}
(m+1)x_1 & \le &
x_{n(m+1)+1}+x_{n(m+1)+2}+\cdots+x_{n(m+1)+m+1} \\
& \le & m\big(y_{n(m+1)+1}+y_{n(m+1)+2}+\cdots+y_{n(m+1)+m+1}\big) =
mz_n,
\end{eqnarray*}
whence $x_1 <_s z_n$ for all $n$. It follows that $x' \le z_0+z_1+
\cdots + z_n$ for some $n$, which entails that $x' \le y_1+y_2+ \dots +
y_{(n+1)(m+1)}$.
\end{proof}

\noindent We shall finally consider the following notion of
$n$-comparison that only involves full elements of the
semigroup. (This shall be appropriate for studying the Corona
Factorization Property for non-simple \Cs s of finite decomposition rank.)

\begin{defi}[Weak $n$- and weak $\omega$-comparison property]
Let $(W,+,\leq)$ be an ordered abelian semigroup. Say that $W$ has
the \emph{weak $n$-comparison} property if whenever $y_0,y_1,\dots,
y_n$ are \emph{full} elements in $W$ and $x \in W$  are such that
$x<_s  y_i$ for all $i$, then $x \leq y_0+y_1+\cdots+ y_n$.

If $W$ is complete, then we say that $W$ has
the \emph{weak $\omega$-comparison} property if whenever $x, x', y_0,y_1,\dots,
$ are elements in $W$  such that $y_0,y_1, \dots$ are
full elements, $x' \ll x$,  and  $x<_s  y_i$ for all $i$, then $x' \leq y_0+y_1+\cdots+ y_n$
for some positive integer $n$.
\end{defi}

\noindent The weak $n$- and the weak $\omega$-comparison properties
only makes sense for
semigroups that contain a full element (if they don't, then they
automatically possess this property). It is clear that if $W$ satisfies
the weak $n$-comparison property for some $n$, then it satisfies the
weak $m$-comparison property for all $m \ge n$ and it satisfies the
weak $\omega$-comparison property.

We show
below that if $W(A)$ has the weak $n$-comparison property then so does
$\Cu(A)$. A new definition and a lemma is required for the proof.

We say that an element $d$ in a \Cs{} $A$ 
is \emph{strictly full} if $(d-\varepsilon)_+$ is
full for some $\varepsilon > 0$, and hence for all sufficiently
small $\varepsilon > 0$. Note that any full projection 
automatically is strictly full. The lemma below characterizes \Cs s
whose primitive ideal space is compact (or, equivalently, \Cs s that
do not contain proper dense algebraic ideals):

\begin{lem}\label{lem_full}
Let $A$ be a $C^*$-algebra such that $A \otimes \cK$ contains a
strictly full positive element. Then every full positive element in
$A$ is strictly full. 
\end{lem}

\begin{proof} View $A$ as a full hereditary sub-\Cs{} of its
  stabilization $A \otimes \cK$. Let $a$ be a strictly full positive
  element in $A \otimes \cK$ and choose $\delta >0$ such that
  $(a-\delta)_+$ is full in $A \otimes \cK$. Let $d$ be an arbitrary 
  full positive element in
  $A$. For each $\ep >0$ consider the closed two-sided ideal
  $I_\ep$ of $A \otimes \cK$ generated by $(d-\ep)_+$. The closure of
  $\bigcup_{\ep > 0} I_\ep$ is a closed two-sided ideal in $A \otimes
  \cK$ which contains $d$ and hence is equal to $A \otimes \cK$. 
  Consequently, $\bigcup_{\ep > 0} I_\ep$ is a dense (algebraic) ideal
  in $A \otimes \cK$, which therefore contains the Pedersen ideal of
  $A \otimes \cK$, which again contains $(a-\delta)_+$. It follows
  that $(a-\delta)_+$ belongs to $I_\ep$ for some $\ep >0$, whence
  $I_\ep = A \otimes \cK$, which 
  again implies that $(d-\ep)_+$ is full (in $A \otimes \cK$ and hence
  in $A$).
\end{proof}

If $A$ is as above, if $x \in \Cu(A)$ is a full element, and if
$\{x_k\}$ is an increasing sequence in $\Cu(A)$ with $x = \sup_k x_k$,
then $x_k$ is full for all sufficiently large $k$. Indeed, there is a
(necessarily strictly full) positive element $a$ in $A \otimes \cK$
such that $x = \langle a \rangle$. Find $\ep > 0$ such that
$(a-\ep)_+$ is (strictly) full. It follows that $x' = \langle
(a-\ep)_+ \rangle$ is full. As $x' \ll x$ we have $x' \le y_k$ for all
large enough $k$, and an element is full if it dominates a full
element.

\begin{lem} \label{wn-comp:W-Cu}
Let $A$ be a \Cs, and assume that
$A \otimes \cK$ contains a strictly full element and that $W(A)$ has the
weak $n$-comparison property. Then $\Cu(A)$ has the weak
$n$-comparison property.
\end{lem}

\begin{proof} We proceed exactly as in the proof of
  Lemma~\ref{n-comp:W-Cu}, but now have that the $y_j$'s in that proof
  are full in $\Cu(A)$. It 
follows from Lemma~\ref{lem_full} above (and the remarks below it) that the
$y_j^{(k)}$'s (from the proof of Lemma~\ref{n-comp:W-Cu})
are full (in $\Cu(A)$ and therefore also in $W(A)$) 
for all large enough $k$. We therefore obtain the desired conclusion
as in the proof of Lemma~\ref{n-comp:W-Cu}.
\end{proof}

\begin{prop}\label{weak_CFP}
Any complete abelian ordered semigroup, which satisfies the weak
$\omega$-com\-pa\-ri\-son property and which contains a full
element that is compactly contained in another (full) element, also
satisfies the Corona Factorization Property.
\end{prop}

\begin{proof} Let
$W$ be a complete abelian ordered semigroup with the weak
  $\omega$-comparison property. By assumption there are elements $v \ll w$
  in $W$  such that $v$ is full.

Let $\{x_k\}$ be a full sequence in
  $W$, let $\{y_k\}$ be another sequence in $W$, let $x' \in W$, and
  let $m$ be a positive integer such that $x' \ll x$ and $x_k \le m
  y_k$ for all $k$. By the definition of a full sequence (applied to $v\ll w$) we have that
  $v \propto x_k$ for all large enough $k$. Hence $v \propto y_k$ for
  all large enough $k$, whence $y_k$ is full whenever $k$ is large
  enough. Upon removing the first finitely many elements from the
  sequences $\{x_k\}$ and $\{y_k\}$ we can assume that all $y_k$ are
  full.

Let now $z_0,z_1, z_2, \dots$ be as in the proof of Proposition~\ref{omega-CFP}
above. Then $z_0,z_1, z_2, \dots$ are full and $x_1 <_s z_j$ for all
$j$. It follows that
$$x' \le z_0+z_1+ \cdots + z_n = y_1+y_2+ \cdots + y_{(n+1)(m+1)}$$
for some $n$, whence $W$ has the Corona Factorization Property.
\end{proof}

\noindent In conclusion, we have defined the following comparability
properties of a complete ordered abelian semigroup, listed in decreasing
strength: $0$-comparison (which is
the same as being almost unperforated), $1$-comparison,
$2$-comparison, $\dots$, $\omega$-comparison, the strong Corona
Factorization Property, and the Corona Factorization Property for
semigroups. Moreover, we have defined weak
$n$- and the weak $\omega$-comparison properties.
We show below that the comparison properties above are in fact
strictly decreasing in strength. An example a complete abelian ordered
semigroup that has the strong Corona 
Factorization Property but not the $\omega$-comparison property is
constructed in \cite{BoRo}.
(That the strong Corona Factorization Property is strictly stronger
than the Corona
Factorization Property was already noted in \cite{OPR}.)

\begin{exem} Let $n$ be a positive integer, let $W_n$ be the
  subsemigroup of $\Z^+$ generated by $\{0,n+1,n+2\}$, and equip $W_n$
  with the algebraic order. Notice that
  $$(n+1)(n+2)-(n+1)-(n+2) = n^2+n-1$$
is the largest natural number that does not belong to
  $W_n$. Suppose that $x,y_0,y_1, \dots, y_n$ belong to $W_n$, that
  $x <_s y_j$ for all $j$, and that $x$ is non-zero.  Then all $y_j$'s
  are non-zero, whence
$$y_0+y_1+ \cdots + y_n - x \ge y_1+y_2+ \dots + y_n \ge n(n+1),$$
(where the ordering above is the usual one in $\Z$), which shows that
$x \le y_0+y_1+ \dots + y_n$ (with respect to the order given on
$W_n$). Hence $W_n$ has the $n$-comparison property.

On the other hand, if we take $x=n+1$ and $y_0=y_1= \dots = y_{n-1} =
n+2$, then $x <_s y_j$ for all $j$, but
$$y_0+y_1+ \dots + y_{n-1}-x = n(n+2)-(n+1) = n^2+n-1,$$
which does not belong to $W_n$. Hence $x \nleq y_0+y_1+ \cdots
+y_{n-1}$ in $W_n$, whence $W_n$ does not have the $(n-1)$-comparison
property.

Next, put $W_\omega = \bigoplus_{n=1}^\infty W_n$ (as an ordered
abelian semigroup). Then $W_\omega$ has the $\omega$-comparison
property, but does not have the $n$-comparison property for any finite
$n$. Indeed, suppose that $x,y^0,y^1, \dots$ in $W_\omega$ are such
that $x <_s y^j$ for all $j$. Write $x = (x_1,x_2, \dots)$ and
$y^j = (y^j_1,y^j_2, \dots)$, with $x_k$ and $y^j_k$ in $W_k$ for all $k$.
Then $x_k = 0$ for all $k$ greater than some $k_0 \in \N$.
Since $W_k$ has $k_0$-comparison when $k \le k_0$, we have
$$x_k \le y^0_k + y^1_k+ \cdots +y^{k_0}_k$$
(in $W_k$) for all $k$, whence $x \le y^0+y^1+ \cdots +y^{k_0}$ (in
$W_\omega$).

Conversely, given a positive integer $n$, choose $x',y'_0,y'_1, \dots,
y'_n$ in $W_{n+1}$ such that $x' <_s y'_j$ for all $j$, and such that
$x' \nleq y'_0+y'_1+ \cdots + y'_n$. (This is possible because
$W_{n+1}$ does not have the $n$-comparison property.)
Let $x,y_0,y_1, \dots, y_n$ in
$W_\omega$ be the elements whose coordinates in the $(n+1)$ position are, respectively,
$x',y'_0,y'_1, \dots, y'_n$, and whose other coordinates
are zero. Then $x <_s y_j$ for all $j$ while $x \nleq y_0+y_1+ \cdots
+ y_n$. Hence $W_\omega$ does not have the $n$-comparison property.
\end{exem}

\section{The Corona Factorization Property for \Cs s with finite
  decomposition rank} \label{sec:CFP}

\noindent
We have already mentioned the result, \cite[Lemma
6.1]{TW}, of Toms and Winter which implies that the Cuntz semigroup
of a simple unital separable \Cs{} has $n$-comparison property. We wish to
extend this result to the non-simple case, and state for this
purpose a lemma whose proof actually is contained in the proof of
\cite[Lemma 6.1]{TW} (follow that proof from Equation (10) to its
end) and therefore is omitted.

\begin{lem}[Toms and Winter] \label{TW_nonsimple}
Let $A$ be a separable $C^*$-algebra with finite decomposition rank
$n$. Suppose that $a,d_0,\ldots, d_n \in A^+$ and $\alpha>0$ satisfy
$$\forall \tau \in T(A): \quad d_\tau(a) < d_\tau(d_i)-\alpha,$$
where $d_\tau$ is the dimension function associated to $\tau$ as in
\eqref{eq:d_tau}. 
Then $a\precsim d_0 \oplus d_1 \oplus \cdots\oplus d_n$.
\end{lem}

\begin{prop}\label{lema_princ}
Let $A$ be a separable, unital $C^*$-algebra with decomposition rank
$n < \infty$. Then $W(A)$ and $\Cu(A)$ have the weak $n$-comparison
property. 
\end{prop}

\begin{proof}
By Lemma~\ref{wn-comp:W-Cu} it suffices to  prove the proposition for
$W(A)$.  

Let $x,y_0,\ldots,y_n\in W(A)$ with $y_i$ full and $x<_s y_i$ be
given for every $i$. Then, by Proposition \ref{stable_comparison},
there exists $k$ such that $(k+1)x \le ky_i$ for all $i$.
%Choose $0
%< \alpha_1 < (k+1)^{-1}$. Then $f(x) < f(y_i) - \alpha_1$ for all
%$i$ and for every state $f$ in $S(W(A),y_i)$. As $y_0,\ldots,y_n$
%are full, and because $\langle 1_A \rangle \ll \langle 1_A \rangle$,
%there is a natural number $N$ such that $Ny_i \ge \langle 1_A
%\rangle$ for all $i$. Put $\alpha_2 = \alpha_1/N$. Then $f(x) <
%f(y_i) - \alpha_2$ for all $i$ and for every state $f$ in
%$S(W(A),\langle 1_A \rangle)$. In particular, with $d_\tau$ denoting
%the (lower semicontinuous) dimension function associated to a
%tracial state $\tau$ on $A$, we have $d_\tau(x) < d_\tau(y_i) -
%\alpha_2$ for all $i$ and for every tracial state $\tau$ on $A$.

As $y_0,\ldots, y_n$ are full, there is a natural number $N$ such that
$\langle 1_A 
\rangle\leq Ny_i$ for all $i$. Choose $0<\alpha'<(k+1)^{-1}$ and let
$\alpha=\alpha'/N$. Let now $f\in S(W(A),\langle 1_A \rangle)$. We
then have that $1\leq Nf(y_i)$ for all $i$, whence
$\alpha<f(y_i)/(k+1)$. Therefore: 
\[ 
f(x)\leq \frac{k}{k+1}f(y_i)=f(y_i)-\frac{f(y_i)}{k+1}<f(y_i)-\alpha\,,
\]
for all $i$.

In particular,  $d_\tau(x) < d_\tau(y_i) -
\alpha$ for all $i$ and for every tracial state $\tau$ on $A$

Finite decomposition rank passes to matrices, so upon replacing $A$
with a matrix algebra over $A$, we can suppose that there exist
positive elements $a$ and $d_0,d_1,\ldots,d_n$ in $A$, with $d_i$
full, such that $x=\langle a \rangle$ and $y_i=\langle d_i \rangle$
for all $i$. Then $d_\tau(a) < d_\tau(d_i) - \alpha$ for all $i$
and for all tracial states $\tau$ on $A$. Lemma~\ref{TW_nonsimple}
then implies that $a \precsim d_0 \oplus d_1 \oplus \cdots \oplus
d_n$, which again implies that $x \le y_0+y_1+ \cdots + y_n$ as
desired.
\end{proof}

\noindent Combining Proposition ~\ref{weak_CFP} and
Proposition~\ref{lema_princ} we get:

\begin{corol} \label{cor:fd}
Let $A$ be a separable, unital $C^*$-algebra with finite decomposition
rank. Then $\Cu(A)$
has the Corona Factorization Property.
\end{corol}

\section{Stability of $C^*$-algebras}
\label{sec:stable}

\noindent We show in this section that a \Cs{} whose Cuntz
semigroup has the $\omega$-comparison property is stable if and only
if it has no unital quotient and no bounded 2-quasitrace. We
introduce a property (S) of a \Cs{} that we show is equivalent to
having no non-zero unital quotients and no bounded 2-quasitraces.

It was shown in \cite{HR} that a separable \Cs{} $A$ is
stable if and only if to every $a \in F(A)$ there exists $b \in A^+$
such that $a \perp b$ and $a \precsim b$. (The set $F(A)$ consists of
all positive elements $a$ in $A$ for which $a = ae$ for some positive
element $e$ in $A$ (that can be taken to be a contraction).)
We shall here consider a
weaker version of this condition, where we replace the
relation $a\precsim b$ with the relation $a \prec_s b$ considered in
Section~\ref{sec:semigroups}.

\begin{defi}
A $C^*$-algebra $A$ is said to have \emph{property (S)} if for every
$a\in F(A)$ there exists $b\in A^+$ such that $a\perp b$ and
$a\prec_s b$.
\end{defi}

\noindent It follows immediately from the definition, the
results from \cite{HR} quoted above, and from Remark~\ref{rem:almunp},
that if $A$ is a separable \Cs{}
for which $\Cu(A)$ is almost unperforated, then $A$ has property (S)
if and only if $A$ is stable. It is easy to see that every stable
\Cs{} has property (S).

\begin{lem}\label{lema00}
Let $A$ be a separable $C^*$-algebra with property (S). Then $A$ has
no non-zero unital quotients.
\end{lem}
\begin{proof}
Let $I$ be an ideal of $A$ such that $A/I$ is unital. Let $e+I$ be
the unit of $A/I$, with $e\in A^+$. Upon replacing $e$ with $g(e)$,
where $g \colon \R^+ \to [0,1]$ is a continuous function which
vanishes on, say $[0,1/2]$, and with $g(1)=1$, we can assume that
$e\in F(A)$. By the assumption that $A$ has property (S) there
exists $b\in A^+$ such that  $e\perp b$ and $e\prec_s b$. Now, $0= eb+I
= b+I$, so $b$ belongs to $I$. The relation $e\prec_s b$
implies that $e$ belongs to the closed two-sided ideal generated by
$b$, and hence to $I$.  Thus, $e+I=0$ and $A/I=0$.
\end{proof}

\begin{lem}\label{lem01} Let $A$ be a separable
$C^*$-algebra with property (S). Then given any $a\in F(A)$ there
exists $b \in F(A)$ such that
$$a\perp b, \qquad a\prec_s b, \qquad a+b\in F(A).$$
If, moreover, $A \otimes \mathcal{K}$ is assumed to contain a strictly
full positive element, then $b$ above can be chosen to be strictly 
full in $A$.
\end{lem}

\begin{proof}
Let $a\in F(A)$, and choose $d$ in $A^+$ with $da=ad=a$. Let $g
\colon \R^+ \to [0,1]$ be a continuous function which is zero on
$[0,1/2]$ and with $g(1)=1$, and put $e=g(d)$. Then
$$e \in F(A), \qquad ea=ae=a, \qquad a \precsim (e-1/2)_+, \qquad \|e\|=1.$$
Since $A$ has property (S) there exists $b_0\in A^+$ such that
$e\perp b_0$ and $e\prec_sb_0$. It follows from Lemma~\ref{ep-delta}
that there exists $\delta >0$ such that $(e-1/2)_+ \prec_s
(b_0-\delta)_+$. Put $b = (b_0-\delta)_+ \in F(A)$, and set
$f= h(b_0)$ where $h \colon \R^+ \to [0,1]$ is a continuous function
such that $h(0)=0$ and $h(t)=1$ for $t \ge \delta$. Then $a \perp
b$, $a \prec_s b$, and
$$(e+f)(a+b) = ea+fb = a+b.$$
The latter shows that $a+b$ belongs to $F(A)$.

Assume now that $A \otimes
\mathcal{K}$ contains a strictly full positive element. We show how to
modify the proof above so that $b$ becomes strictly full.
Let $B$ be the hereditary sub-\Cs{} of $A$ consisting
of all elements which are orthogonal to $e$. Then $B$ is full in
$A$. Indeed, because $e$ belongs to $F(A)$ there exists a positive
element $e'$ in $A$ such that $e'e=ee'=e$. Let $I$ be the closed
two-sided ideal in $A$ generated by $B$, and assume, to reach a
contradiction, that $I$ were proper. Then $e'+I$ would be a unit for
$A/I$, thus contradicting Lemma~\ref{lema00}.

It follows from Brown's theorem that $B \otimes \cK$ is isomorphic
to $A \otimes \cK$, and so $B \otimes \cK$ contains a strictly full
positive element. Hence, by Lemma~\ref{lem_full}, any full element in $B$
is strictly full. Upon replacing the element $b_0$ from above with the
sum of $b_0$ and a positive strictly full element in $B$ we can assume
that $b_0$ is strictly full. It follows that $b
=(b_0-\delta)_+$ is full (and hence strictly full) if $\delta > 0$
is chosen sufficiently small.
\end{proof}

\begin{lem}\label{rema_ortho}
Let $A$ be a separable $C^*$-algebra with property (S). Then for
every $a \in F(A)$ there is a sequence $b_0,b_1,b_2, \dots$ of
elements in $F(A)$ such that the elements $a,b_0,b_1,b_2, \dots$ are
pairwise orthogonal, $a+b_0+b_1+ \cdots + b_n$ belongs to $F(A)$ for
all $n$, and such that $a\prec_s b_0 \prec_s b_1 \prec_s \cdots$.

If, moreover, $A\otimes \mathcal{K}$ is assumed to contain a strictly 
full positive element, then $b_0,b_1,b_2, \dots$ above can be chosen to be
strictly full in $A$.
\end{lem}

\begin{proof} The existence of $b_0$ such that $a \perp b_0$,
  $a\prec_s b_0$, and $a+b_0$ belongs to $F(A)$ follows from
  Lemma~\ref{lem01}. Suppose that $n \ge 0$ and that
$b_0,b_1,\dots,b_n$ have been found such that $a,b_0,b_1, \dots,b_n$
are pairwise orthogonal,  $a\prec_s b_0 \prec_s b_1 \prec_s \cdots
\prec_s b_n$, and $a+b_0+b_1+ \cdots + b_n$ belongs to $F(A)$. Then,
by Lemma~\ref{lem01}, there is $b_{n+1}$ in $F(A)$ which is
orthogonal to the sum $a+b_0+b_1+ \cdots + b_n$ (and hence to each
of the summands), such that $a+b_0+b_1+ \cdots + b_n \prec_s
b_{n+1}$ (and hence $b_n \prec_s b_{n+1}$) and such that $a+b_0+b_1+
\cdots + b_{n+1}$ belongs to $F(A)$.

Finally, use Lemma~\ref{lem01} to see that each of the positive
elements $b_j$ above can be chosen to be strictly full if $A \otimes
\cK$ contains a strictly full positive element.
\end{proof}

\noindent We will now give an algebraic characterization of property
(S) for a $C^*$-algebra. The characterization is very similar to,
but sharpens, \cite[Theorem~3.6]{HRW}. The reader is referred to
\cite{BH} for the definition and properties of 2-quasitraces. Let
us just here remind the reader than any 2-quasitrace on an exact
\Cs{} is a trace, and that the short\-coming of a quasitrace
(compared with a trace) is that it only is assumed to be additive on
commuting elements.

\begin{prop}\label{prop1}
Let $A$ be a separable $C^*$-algebra. Then $A$ has property (S) if
and only if $A$ has  no non-zero bounded lower semi-continuous 
$2$-quasitrace and no non-zero unital quotient.
\end{prop}

\begin{proof}
The ``if'' part is contained in the proof of \cite[Theorem
3.6]{HRW}.

To prove the  ``only if'' part, suppose that $A$ has property (S).
By Lemma~\ref{lema00}, $A$ has no non-zero unital quotients.
Suppose, to reach a contradiction, that $\tau$ is a non-zero bounded
$2$-quasitrace on $A$, and let $d_\tau$ be the associated lower
semicontinuous dimension function on $W(A)$ (cf., Equation 
\eqref{eq:d_tau}). Since $\tau$ is 
non-zero there is a positive element $a$ in $A$ such that
$d_\tau( a ) >0$, and since $d_\tau$ is lower
semicontinuous, $d_\tau((a-\varepsilon)_+) > 0$ for
some $\ep >0$. We can now use Lemma~\ref{rema_ortho} to find a
sequence $b_0 = (a-\ep)_+,b_1,b_2, \dots$ of pairwise orthogonal
elements in $F(A)$ such that $b_0\prec_s b_1 \prec_s b_2 \prec_s
\cdots$. By Proposition~\ref{stable_comparison} we have $0 <
d_\tau(\langle b_0 \rangle) < d_\tau(\langle b_1 \rangle) <
d_\tau(\langle b_2 \rangle) < \cdots$, and in particular
$$d_\tau(\langle b_0+b_1+ \cdots + b_n \rangle) \ge (n+1)
d_\tau(\langle (a-\ep)_+ \rangle).$$ On the other hand, one has
$d_\tau(\langle c \rangle) \le \|\tau\|$ for all $c$ in $A^+$, and
so the inequality above is in contradiction with the assumed
boundedness of $\tau$.
\end{proof}

\noindent It is well-known that stability is not a stable property
(see \cite{Ro2}). Property (S), however, is a stable property, as
easily follows from Proposition~\ref{prop1} above:

\begin{corol} \label{cor:S-stable}
Let $A$ be a separable \Cs. Then the following conditions are
equivalent:
\begin{enumerate}
\item $A$ has property (S).
\item $M_n(A)$ has property (S) for some natural number $n$.
\item $M_n(A)$ has property (S) for all natural numbers $n$.
\end{enumerate}
\end{corol}

\begin{proof} By Proposition~\ref{prop1} it suffices to check that each of
  the two properties:
  having a non-zero unital quotient, and having a non-zero bounded
  2-quasitrace, passes to matrix algebras and back again. This is
  trivial for the first. It is a theorem (see \cite{BH}) that
  2-quasitraces extend to all matrix algebras (and vice versa).
\end{proof}

\noindent The result, \cite[Theorem 3.6]{HRW}, that we have used extensively
in the proof of Proposition~\ref{prop1} above, actually says that a
separable \Cs{} with almost unperforated Cuntz semigroup is stable
if and only if it has no non-zero unital quotient and no non-zero
bounded 2-quasitrace. Reminding the reader that almost
unperforation is the same as the ``0-comparison'' property, we can
extend \cite[Theorem 3.6]{HRW} as follows:

\begin{prop}\label{prop2}
Let $A$ be a separable $C^*$-algebra such that $\Cu(A)$ satisfies the
$\omega$-comparison property (cf.\ Definition~\ref{def:wcomp}). Let
$B$ be a hereditary sub-\Cs{} of $A \otimes \cK$. Then the following
conditions are equivalent:
\begin{enumerate}
\item $B$ is stable
\item $B$ has no non-zero unital quotient and no non-zero bounded
  2-quasitrace.
\item $B$ has property (S).
\end{enumerate}
\end{prop}

\begin{proof} Conditions (ii) and (iii) are equivalent for all
  separable \Cs s by Proposition~\ref{prop1}, and (i) clearly
  implies (ii) (again for all \Cs s) (see eg.\ \cite{HR}).

(iii) $\Rightarrow$ (i). By \cite[Proposition~2.2]{HR} it is enough
to show that for every $a\in B^+$ and every $\varepsilon>0$ there
exists $b\in B^+$ such that $(a-\varepsilon)_+\precsim b$ and
$(a-\varepsilon)_+\perp b$. (Indeed, if such an element $b$ exists,
then $(a-2\ep)_+ = x^*bx$ for some $x \in B$; whence $(a-2\ep)_+
\sim b^{1/2}xx^*b^{1/2} := b_0 \perp (a-2\ep)_+$, and
$\|a-(a-2\ep)_+\| \le 2\ep$.)

Since $(a-\varepsilon/2)_+$ belongs to $F(B)$ we can apply
Lemma~\ref{rema_ortho} to get a sequence of positive elements $b_0,
b_1,b_2, \dots$ in $F(B)$ such that $(a-\ep/2)_+ \prec_s b_0 \prec_s
b_1 \prec_s b_2 \prec_s\cdots$ and for which $(a-\ep/2)_+,
b_0,b_1,b_2,\dots $ are mutually orthogonal. The corresponding elements
$$x' = \langle (a-\ep)_+ \rangle, \qquad x = \langle (a-\ep/2)_+
\rangle, \qquad y_j = \langle b_j \rangle$$
in $\Cu(A)$ satisfy $x' \ll x$ and $x <_s y_j$ for all $j$. By the
assumption that $\Cu(A)$ satisfies 
the $\omega$-comparison property there is a natural number $n$ such
that $x' \le y_1 + \cdots + y_n$.  Thus, $(a-\varepsilon)_+\precsim
b_0+\cdots + b_n$ (relatively to $A \otimes \cK$ and therefore also
relatively to the hereditary sub-\Cs{} $B$ of $A \otimes \cK$) and 
$(a-\varepsilon)_+\perp b_0+\cdots + b_n$ as desired.
\end{proof}

\noindent We have the following analog of Proposition~\ref{prop2},
where the assumption on the comparison property of the Cuntz
semigroup is weakened, but where we instead have to assume the
existence of a full projection:

\begin{prop}\label{propw2}
Let $A$ be a separable $C^*$-algebra such that $\Cu(A)$ satisfies the
weak $\omega$-comparison property and such that $A \otimes \cK$
contains a full projection. Let $B$ be a full hereditary sub-\Cs{}
of $A \otimes \cK$. Then the following conditions are equivalent:
\begin{enumerate}
\item $B$ is stable
\item $B$ has no non-zero unital quotient and no non-zero bounded
  2-quasitrace.
\item $B$ has property (S).
\end{enumerate}
\end{prop}

\begin{proof} Proceeding as in the proof of Proposition~\ref{prop2},
we only need to prove
(iii) $\Rightarrow$ (i); and to prove that (i) holds it suffices to
show that  for every $a\in B^+$ and every $\varepsilon>0$ there exists
$b\in B^+$ such that $(a-\varepsilon)_+\precsim b$ and
$(a-\varepsilon)_+\bot b$.

Since $A \otimes \cK$ contains a full projection (which automatically is
strictly full), it follows from Lemma~\ref{lem_full} that all full positive 
elements in $A \otimes \cK$ are strictly full. As $B$ is separable, it
contains a full (positive) element, which is then full in $A \otimes
\cK$ and hence is strictly full in $A \otimes \cK$ and therefore also
strictly full in $B$. 
Apply Lemma~\ref{rema_ortho} to get \emph{full} positive elements
$b_0, b_1,b_2, \ldots $ in $F(B)$ such that $(a-\ep/2)_+ \prec_s b_0
\prec_s b_1 \prec_s b_2 \prec_s \cdots$ and such that $(a-\ep/2)_+, b_0,b_1,b_2
\dots$ are mutually orthogonal. Proceed as in the proof of
Proposition~\ref{prop2} (where it suffices to have 
the weak $\omega$-comparison property because the elements $y_j =
\langle b_j \rangle$ are full in $\Cu(A)$) to obtain the desired
element $b$.
\end{proof}

\noindent
We end this section
by describing when separable \Cs s with finite decomposition
rank are stable (under the assumption that their stabilization
contains a full projection):

\begin{corol}\label{corol2}
Let $A$ be a separable $C^*$-algebra with finite decomposition rank,
and assume that $A \otimes \cK$ contains a full projection. Then the
following conditions are equivalent:
\begin{enumerate}
\item $A$ is stable.
\item $A$ has no non-zero unital quotients and no non-zero bounded
positive traces.
\item $A$ has property (S).
\end{enumerate}
\end{corol}

\begin{proof}
Let $p$ be a full projection in $A \otimes \cK$ and put $B = p(A
\otimes \cK)p$. Then $B$ has the same decomposition rank as $A$, say
$n$; and $B$ is unital. It follows from Proposition~\ref{lema_princ}
that $\Cu(B)$ has weak $n$-comparison and therefore also the weak
$\omega$-comparison property. As $A$ is (isomorphic to) a full
hereditary sub-\Cs{} of $B \otimes \cK$ the result now follows from
Proposition~\ref{propw2}. (In (ii) we have used that any
2-quasitrace on a nuclear \Cs{} is a trace.)
\end{proof}

\section{The Corona Factorization Property and the Cuntz semigroup} \label{sec:CFP_Cuntz}

\noindent Recall that a $C^*$-algebra $A$ is said to have the
\emph{Corona Factorization Property} if every full projection in the
multiplier algebra of $A\otimes \mathcal{K}$ is properly infinite.
It was observed by Kucerovsky and Ng in \cite{KN} that the Corona
Factorization Property is equivalent to a 
statement regarding stability of full hereditary sub-\Cs s of the
stabilized \Cs. Our
aim here is to characterize the Corona Factorization Property for
$C^*$-algebras in terms of the comparison property of the
Cuntz semigroup of the same name introduced in
Section~\ref{sec:semigroups}. We shall need several lemmas before we
can arrive at the main results of this section.

%First recall from \cite{HR} that $F(A)$ is the set of \emph{compactly supported} positive elements %of $A$, i.e., the set
%of all $a\in A^+$ such that there exists $e\in A^+$ with $ea=ae=a$.

\begin{lem}\label{approx}
Let $A$ be a $\sigma$-unital $C^*$-algebra and suppose that $\{e_k\}$
is an increasing approximate unit for $A$ consisting of positive
contractions. Then:
\begin{enumerate}
\item For every positive $a$ in $A$ and for every $\ep>0$ one has
  $(a-\ep)_+ \precsim e_k$ for all large enough $k$.
\item  $\{\langle e_k\rangle\}$ is a full sequence in $\Cu(A)$.
\end{enumerate}
\end{lem}

\begin{proof}
(i). We have $\|a^{1/2}e_ka^{1/2}-a\| < \ep$ for $k$ large enough,
  whence $(a-\ep)_+ \precsim a^{1/2}e_ka^{1/2} \precsim e_k$.

(ii). The sequence $\{\langle e_k\rangle\}$ is clearly increasing. The
fullness property of this sequence follows from (i), from the fact
that $\{e_k \otimes 1_n\}_{k=1}^\infty$ is an approximate unit for
$M_n(A)$ and that $\bigcup_{n=1}^\infty M_n(A)$ is dense in $A\otimes \cK$.
\end{proof}

\noindent
Recall from Section~\ref{sec:stable} the definition of the set $F(A)$ of
compactly supported elements in a \Cs{} $A$. Suppose that $A$ is
$\sigma$-unital. Then, to any strictly positive element $c$ in $A$ one
can associate the set
\[
F_c(A):=\{b\in A^+ \mid g_{\varepsilon}(c)b=b\text{ for some }
\varepsilon
>0\},
\]
cf.\ \cite{HR}, where $g_\ep \colon \R^+ \to \R^+$ is the continous
function that vanishes on $[0,\ep]$, is equal to $1$ on
$[2\ep,\infty)$ and is linear on $[\ep,2\ep]$. 
It is easy to see that $F_c(A)$ is a dense subset of
$F(A)$, which---unlike $F(A)$---is closed under addition.

We shall use below that whenever $c\in A$ is a strictly positive
element of $A$, then $c\otimes 1_n$ is a strictly positive element of
$M_n(A)$.

\begin{lem}
\label{lem:diagonal} Let $c$ be a strictly positive element of a
$C^*$-algebra $A$, and let $a = (a_{ij})$ be a positive element in
$M_n(A)$. Let $d = \sum_{j=1}^n a_{jj} \in A^+$ be the sum of the
diagonal elements of $a$. Then $\langle a\rangle\leq n\langle
d\rangle$; and $d$ belongs to $F_c(A)$ if $a$ belongs to $F_{c\otimes
  1_n}(M_n(A))$.
\end{lem}

\begin{proof}
Let $\ep>0$.
For each $i=1,2,\dots, n$, let $\{e_k^{(i)}\}_{k=1}^\infty$ be an
approximate unit for $\overline{a_{ii}Aa_{ii}}$, and put $e_k= \text{diag}(e^{(1)}_k,\ldots,e^{(n)}_k)$. Then  $e_k^{(i)} \precsim
a_{ii} \precsim d$ for all $k$. Also, $\{e_k\}$ is an
approximate unit for $\overline{aM_n(A)a}$, whence $(a-\ep)_+ \precsim
e_k$ for all large enough $k$, cf.\ Lemma~\ref{approx}.
We therefore conclude that
$$\langle (a-\ep)_+ \rangle \le \langle e_k \rangle = \sum_{i=1}^n
\langle e_k^{(i)} \rangle \le n \langle d \rangle.
$$
This proves the first claim because $\ep >0$ was arbitrary.

Suppose that $a$ belongs to $F_{c\otimes
  1_n}(M_n(A))$. Then there is
$\varepsilon>0$ with $g_{\varepsilon}(c\otimes 1_n)a=a$. As
$g_{\varepsilon}(c\otimes 1_n)=g_{\varepsilon}(c)\otimes 1_n$, this
is easily seen to imply that $g_{\varepsilon}(c)a_{ii}=a_{ii}$ for
all $i$, hence $a_{ii}\in F_c(A)$. Thus $d$ belongs to
$F_c(A)$.
\end{proof}

\begin{lem}\label{lem:spectacular} Let $A$ be a $\sigma$-unital
  $C^*$-algebra, and fix a strictly positive element $c$ in $A$.
Suppose that $M_n(A)$ is stable for
  some positive integer $n$. Then for all elements $a,b$ in $F_c(A)$ there
  exists an element $d$ in $F_c(A)$ with $a\perp d$ and $b \otimes 1_n
  \precsim d\otimes 1_n$.
\end{lem}

\begin{proof}
Let $a,b\in F_c(A)$. Then $a$ and $b$ both belong to
$\overline{g_\delta(c)Ag_\delta(c)}$ for some $\delta>0$. 
Clearly, $a,b\precsim
g_\delta(c)$ and $g_\delta(c)\otimes 1_n\in F_{c\otimes 1_n}(M_n(A))$. Using
that $M_n(A)$ is stable, we find an element $b'\in F_{c\otimes
1_n}(M_n(A))$ such that $b'\perp g_\delta(c)\otimes 1_n$ and
$g_\delta(c)\otimes 
1_n\precsim b'$, cf.\ \cite[Lemma~2.6~(i)]{HR}. Let $d \in A$ be the sum of the
diagonal elements in $b'$. Then
$d$ belongs to $F_c(A)$ and $b'\precsim d\otimes 1_n$ by
Lemma~\ref{lem:diagonal}. This shows that
$$b \otimes 1_n \precsim g_\delta(c) \otimes 1_n \precsim b' \precsim d
\otimes 1_n.$$
Since $b'\perp g_\delta(c) \otimes 1_n$, it follows that $d\perp g_\delta(c)$, whence
$d \perp a$.
\end{proof}

\noindent
The lemma below is a reformulation of the characterization of
stability  from \cite{HR}.

\begin{lem}\label{lem:tobeproved}
Let $A$ be a $\sigma$-unital
  $C^*$-algebra with a
strictly positive element $c$. Then $A$ is stable if and only if for
every $\varepsilon>0$ there exists $b\in A^+$ such that $b\perp
(c-\varepsilon)_+$ and $(c-\varepsilon)_+\precsim b$.
\end{lem}

\begin{proof}
The ``only if'' part follows from \cite[Theorem~2.1]{HR}. To prove the
``if'' part, we verify that condition (b) of
\cite[Proposition~2.2]{HR} is satisfied. To this end, let $a \in F(A)$
and $\ep>0$ be given. Choose $\delta >
0$ such that $\|a - g_\delta(c)ag_\delta(c)\| < \ep/2$, put $a' =g_\delta(c)ag_\delta(c)$,  and find $d \in
A^+$ such that
$(c-\delta)_+ \perp d$ and $(c-\delta)_+ \precsim  d$. Then
$a' \perp d$ and $a' \precsim g_\delta(c) \precsim d$. 
Hence there exists $t$ in $A$ such that
$b':=(a'-\ep/2)_+=t^*dt$. Put $c' =
d^{1/2}tt^*d^{1/2}$. Then $\|a-b'\| \leq \ep$, $b' \perp c'$, and $b'
\sim c'$.
\end{proof}

\begin{prop}\label{prop:CFP_stable}
Let $A$ be a $\sigma$-unital $C^*$-algebra whose Cuntz semigroup $\Cu(A)$
satisfies the Corona Factorization Property for semigroups.
Then $A$ is stable if $M_m(A)$ is stable for some $m\in
\mathbb{N}$.
\end{prop}

\begin{proof}
Suppose that $M_m(A)$ is stable for some natural number $m$. Let $c$
be a strictly positive element in $A^+$, and let $\varepsilon>0$ be
given. Choose a decreasing sequence $\{\varepsilon_n\}$ of strictly
positive real numbers that converges to zero, and such that
$\varepsilon_1<\varepsilon$. Let $a_n=(c-\varepsilon_n)_+$. Since
$a_n$ is Cuntz equivalent to 
$g_{\varepsilon_n}(c)$ and $\{g_{\varepsilon_n}(c)\}$ is
an increasing approximate unit for $\overline{cAc}=A$, it follows
from Lemma \ref{approx} that $\{\langle a_n\rangle\}$ is a full sequence
in $\Cu(A)$.

We use Lemma \ref{lem:spectacular} to construct a sequence
$d_1,d_2,d_3 \dots$ of positive elements in $F_c(A)$ such that
$a_1,d_1,d_2, \dots$ are pairwise orthogonal and $a_n \precsim d_n
\otimes 1_m$ for all $n$. Indeed, at stage $n$, since $a_1,d_1, \dots,
d_{n-1}$ belong to $F_c(A)$, so does their sum, and so we can find $d_n \in
F_c(A)$ orthogonal to $a_1+d_1+ \cdots +d_{n-1}$ satisfying $a_n
\precsim d_n \otimes 1_m$.

Now, apply the Corona Factorization Property for $\Cu(A)$ to
$\{\langle a_n\rangle\}$ and $\{\langle d_n\rangle\}$ (that
satisfies $\langle a_n\rangle\leq m\langle d_n\rangle$ for all $n$).
Thus, for our $\varepsilon>0$, since $\langle (c-\varepsilon)_+\rangle\ll \langle
(c-\varepsilon_1)_+\rangle=\langle a_1\rangle$ there is a natural number $k$ such
that
\[
\langle (c-\varepsilon)_+\rangle\leq
\langle d_1\rangle+\langle d_2\rangle+ \cdots + \langle d_k\rangle =
\langle d_1+d_2+ \cdots + d_k\rangle\,,
\]
which implies that $A$ is stable (by virtue of Lemma~\ref{lem:tobeproved}).
\end{proof}

\noindent
If $A$ is a non-unital \Cs{}, then we shall denote its
multiplier algebra by $\mathcal{M}(A)$ and the unit in the multiplicer
algebra by $1$.

\begin{lem} \label{lm:00}
Let $A$ be a $\sigma$-unital stable \Cs{} and let $a$ be a positive
contraction in $A$. Then $1 -a$ is a properly infinite, full,
positive element in $\mathcal{M}(A)$.
\end{lem}

\begin{proof}  It follows from \cite[Corollary 4.3]{HR} that
  $\overline{(1-a)A(1-a)}$ is stable. Hence $1-a$ is
  properly infinite, cf.\ \cite[Proposition 3.7]{KR}. We proceed to
  prove that $1-a$ is full in $\mathcal{M}(A)$.

  Take positive functions $f,g \colon [0,1]
  \to [0,1]$ such that $f$ is zero on $[0,1/2]$, $f+g=1$, and
  $g(1)=0$. Then $g(a)$ belongs to $\overline{(1-a)A(1-a)}$. Since $A$
  is stable and $\sigma$-unital we can find a positive element $b$ in
  $A$ such that $b \perp (a-1/2)_+$ and $(a-1/2)_+ \precsim b$. Then
  $f(a) \perp b$, whence
$$b = \big(f(a) + g(a)\big)b\big(f(a)+g(a)\big) = g(a)bg(a) \in
 \overline{(1-a)A(1-a)}.$$
As $f(a) \precsim (a-1/2)_+ \precsim b$, we see that $f(a)$ belongs
to the closed two-sided ideal in $\mathcal{M}(A)$ generated by
$1-a$. As $g(a)$ belongs to $\overline{(1-a)\mathcal{M}(A)(1-a)}$,
we conclude that the closed two-sided ideal generated by $1-a$
contains $1 = f(a)+g(a)$, and hence is equal to $\mathcal{M}(A)$.
\end{proof}

\begin{lem} \label{lm:0}
Let $A$ be a $\sigma$-unital stable \Cs{} and let $T$ be a positive
element in $\mathcal{M}(A)$ such that $1 \precsim T$ (or,
equivalently, such that $T$ is full and properly infinite). Then
$\overline{TAT}$ is stable.
\end{lem}

\begin{proof} Put $B = \overline{TAT}$. Since $A$ is $\sigma$-unital,
  then so is $B$.

There is $\delta > 0$ such that $1 \precsim
  (T-2\delta)_+$, whence $1 = R^*(T-\delta)_+R$ for some element $R$
  in $\mathcal{M}(A)$. Put $V = (T-\delta)_+^{1/2}R$ and put $T' = g(T)$,
  where $g\colon \R^+ \to [0,1]$ is a continuous function such that
  $g(0)=0$, $g(t)=1$ for $t \ge \delta$, and $g$ is linear on
  $[0,\delta]$. Then $\overline{T'AT'} = \overline{TAT} =B$, and $V$ is
  an isometry whose range projection satisfies $VV^*T' = VV^*$.

To show that $B$ is stable, we use \cite{HR}, by which it suffices
to show that for each $a \in F(A)$ there is $b \in A^+$ such that $a
\perp b$ and $a \sim b$. Take $a \in F(B)$, and let $e$ be a
positive
  contraction in $B$ such that $ae = ea = a$. Put $T_0 =
(1-e)T'(1-e)$, and note that $\overline{T_0AT_0} \subseteq B$. Now,
$$V^*T_0V = V^*T'V - V^*(eT'+T'e-eT'e)V = 1 - c,$$
with $c = V^*(eT'+T'e-eT'e)V \in A$. As $V^*T_0V$ is a positive
contraction, the element $c$ is also a positive contraction. We can
therefore use Lemma~\ref{lm:00} to conclude that $V^*T_0V$ is
properly infinite and full. As $V^*T_0V \precsim T_0$ we also have
that $T_0$ is properly infinite and full. This again entails that $1
\precsim T_0$, and so there is an isometry $W$ in $\mathcal{M}(A)$
whose range projection, $WW^*$, belongs to
$\overline{T_0\mathcal{M}(A)T_0}$. In particular, $WW^* \perp a$.
Put $b=WaW^*$. Then $b$ is a positive element in $B$, $b \perp a$,
and $b \sim a$ as desired.
\end{proof}

\noindent Let $A$ be a stable \Cs. Then there exists a sequence
$\{S_n\}$ of isometries in $\mathcal{M}(A)$ with orthogonal range
projections and such that $\sum_{n=1}^\infty S_nS_n^* = 1$ (the sum
being convergent in the strict topology). Let $\{a_n\}$ be any
bounded sequence of elements in $A$ (or in $\mathcal{M}(A)$). Then
$\sum_{n=1}^\infty S_n a_n S_n^*$ is strictly convergent to an
element in $\mathcal{M}(A)$. We shall denote this element by
$\bigoplus_{n=1}^\infty a_n$. If $\{T_n\}$ is another sequence of
isometries in $\mathcal{M}(A)$ with range projections adding up to 1
in the strict topology, then $\sum_{n=1}^\infty T_nS_n^*$ is
strictly convergent to a unitary $U$ in $\mathcal{M}(A)$ and
$U\Big(\sum_{n=1}^\infty S_na_nS_n^*\Big)U^* = \sum_{n=1}^\infty
T_na_nT_n^*$. This shows that the element $\bigoplus_{n=1}^\infty
a_n$ is independent on the choice of the sequence $\{S_n\}$ of
isometries, up to unitary equivalence.

\begin{lem} \label{lm:1}
Let $A$ be a stable $\sigma$-unital \Cs{} which satisfies the Corona
  Factorization Property. Let $T$ be a full, positive element in
  $\mathcal{M}(A)$. Then $a \precsim T$ for every positive element $a$ in $A$.
\end{lem}

\begin{proof} Put $B=\overline{TAT}$. Then $B$ is a full hereditary
  sub-\Cs{} of $A$ because $T$ is full in $\mathcal{M}(A)$.

Again using that $T$ is a full element in the properly infinite
\Cs{} $\mathcal{M}(A)$, there is a positive integer $n$ such
  that $T \otimes 1_n$ is properly infinite.  As,
$$M_n(B) = \overline{(T \otimes 1_n)M_n(A)(T \otimes 1_n)},$$
we conclude from Lemma~\ref{lm:0} that $M_n(B)$ is stable. Because
$A$ is assumed to satisfy the Corona Factorization Property, we can
now conclude from \cite[Theorem 4.2]{KN} that $B$ is stable.

Let $a$ be a positive element in $A$ and let $\ep > 0$ be given. As
$B$ is full in $A$ we can find a positive integer $n$, positive
elements $b_1, \dots, b_n$ in $B$, and elements $x_1, \dots, x_n$ in
$A$ such that $(a-\ep)_+ = \sum_{j=1}^n x_j^*b_jx_j$. Because $B$ is
stable there are isometries $S_1, \dots, S_n$ in $\mathcal{M}(B)$
with orthogonal range projections. We now get
$$(a-\ep)_+ \, \precsim \, b_1 \oplus b_2 \oplus \cdots \oplus b_n \,
\approx \, S_1b_1S_1^* + S_2b_2S_2^* + \cdots + S_nb_nS_n^* \,
\precsim \, T.$$ As this holds for all $\ep >0$, we have $a \precsim
T$ as desired.
\end{proof}

The lemma below is similar to \cite[Corollary 2.7]{Ror:JOT}, but we
do not assume below that $A$ is unital. If $a$ and $b$ are positive
elements in a \Cs{} and if $m$ is a positive integer, then we shall
write $a \precsim_m b$ to denote that $a \precsim b \otimes 1_m$.

\begin{lem} \label{lm:2} Let $A$ be a $\sigma$-unital
stable \Cs, and let $c$ be a strictly positive contraction in $A$.

Let $\{a_n\}$
  be a bounded sequence of positive elements in $A$. Then
  $\bigoplus_{n=1}^\infty a_n$ defines a full element in $\mathcal{M}(A)$ if
  there exist $\delta >0$ and a positive integer $m$ such that for
 every $\ep>0$ and for every positive integer $k$ there
 is an integer $\ell >k$ such that
$$(c-\ep)_+ \precsim_m (a_k-\delta)_+ \oplus (a_{k+1} -\delta)_+
\oplus \cdots \oplus (a_\ell -\delta)_+.$$
\end{lem}

\begin{proof} We show that $1 \precsim_m \bigoplus_{n=1}^\infty a_n$,
  which of course will imply that $\bigoplus_{n=1}^\infty a_n$ is
  full. By assumption we can find integers $1=k_1 < k_2 < k_3 < \cdots$ such
that
$$(c-\tfrac{1}{n})_+  \precsim_m (a_{k_n}-\delta)_+ \oplus (a_{k_n+1}-\delta)_+
\oplus \cdots \oplus (a_{k_{n+1}-1}-\delta)_+$$ for all $n$.

Choose isometries $T_1, T_2, \dots, T_m$ in $\mathcal{M}(A)$ with
  range projections adding up to 1. Then we can identify
  $\big(\bigoplus_{n=1}^\infty a_n\big) \otimes 1_m$ with
$$\sum_{j=1}^m T_j \Big(\bigoplus_{n=1}^\infty a_n\Big) T_j^*
\, = \, \sum_{n=1}^\infty \sum_{j=1}^m T_jS_na_nS_n^*T_j^* \, \sim
\, \sum_{n=1}^\infty S_n\Big(\sum_{j=1}^m T_ja_nT_j^* \Big)S_n^*.$$
(We have here used that the range projections of the two families of
isometries, $\{S_nT_j\}$ and $\{T_jS_n\}$, sum to 1 in the strict
topology.) Put
$$b_n = \sum_{k=k_n}^{k_{n+1}-1} S_k \Big(\sum_{j=1}^m
T_ja_nT_j^*\Big) S_k^* \sim \big(a_{k_n} \oplus a_{k_n+1} \oplus
\cdots \oplus a_{k_{n+1}-1}\big) \otimes 1_m.$$ Then
$\big(\bigoplus_{n=1}^\infty a_n\big) \otimes 1_m \sim
\sum_{n=1}^\infty b_n$, the latter sum is strictly convergent, and
$(c-\tfrac{1}{n})_+ \precsim (b_n-\delta)_+$ for all $n$.
We must show that $1 \precsim \sum_{n=1}^\infty b_n$.

Choose a strictly decreasing sequence $\{\delta_n\}$ of positive
real numbers such that $\delta_2 = 1$ and $\delta_{n+2} > 1/n$ for
all $n$. Define $g_n \colon [0,1] \to [0,1]$ to be the continuous
function which is zero on $[0,\delta_{n+2}] \cup [\delta_n,1]$ (note
that $[\delta_1,1]=\emptyset$), $g_n(\delta_{n+1})=1$, and $g_n$ is
linear on $[\delta_{n+2},\delta_{n+1}]$ and on
$[\delta_{n+1},\delta_{n}]$. Then $1 = \sum_{n=1}^\infty g_n(c)$ and
the sum is strictly convergent. Moreover, since $\delta_{n+2} > 1/n$,
we have $g_n(c) = x_n^*(b_n-\delta)_+x_n$ for some element $x_n$ in
$A$. Let $h \colon [0,1] \to \R^+$ be the continuous function which
satisfies $h(0)=0$, $h(t) = t^{-1/2}$ for $t \ge \delta$, and $h$ is
linear on $[0,\delta]$. Put $y_n = h(b_n)(b_n-\delta)_+^{1/2}x_n$.
Then $\|y_n\| \le \delta^{-1/2}$ (because
$\|(b_n-\delta)_+^{1/2}x_n\| = \|g_n(c)\|^{1/2} =1$ and $\|h(b_n)\| \le
\delta^{-1/2}$), and $y_n^*b_ny_n = g_n(c)$. Notice that $y_n$
belongs to the set $\overline{b_nAg_n(c)}$. Put $Y = \sum_{n=1}^\infty
y_n \in \mathcal{M}(A)$ (the sum is strictly
convergent). Then,
$$Y^* \Big(\sum_{n=1}^\infty b_n\Big) Y = \sum_{n=1}^\infty Y^*b_nY
= \sum_{n=1}^\infty y_n^*b_n y_n = \sum_{n=1}^\infty g_n(c) = 1,$$
which shows that $1 \precsim \sum_{n=1}^\infty b_n$.
\end{proof}

\begin{lem} \label{lm:3}
Let $A$ be a stable $\sigma$-unital \Cs{} which satisfies the Corona
Factorization Property. Let $a_1,a_2, \dots, b_1,b_2, \dots$ be
positive elements in $A$, and let $m$ be a positive integer such
that $a_1 \precsim a_2 \precsim a_3 \precsim \cdots$, such that the
set $\{a_n\}$ is full in $A$, and such that $a_n \precsim
b_n \otimes 1_m$ for all $n$.
It follows that for each $\eta>0$ there is a natural number $k$ such
that
$$(a_1-\eta)_+ \precsim b_1 \oplus b_2 \oplus \cdots \oplus b_k.$$
\end{lem}

\begin{proof} We note first that we can choose $\delta_n > 0$ such
  that
$$(a_1-\delta_1)_+ \precsim (a_2-\delta_2)_+ \precsim
  (a_3-\delta_3) \precsim \cdots,$$
and such that $\{(a_n-\delta_n)_+\}_{n=1}^\infty$ is full in $A$.
(Let us prove this fact: As $a_j \precsim a_n$ whenever $1 \le j <
n$ there is $\eta_n > 0$ such that $(a_j - 1/n)_+ \precsim (a_n -
\eta_n)_+$ for $j=1,2, \dots, n-1$. We choose now $\delta_n$
inductively such that $0 < \delta_n \le \eta_n$ and such that
$(a_{n-1}-\delta_{n-1})_+ \precsim (a_{n}-\delta_{n})_+$. For $n=1$
we can take $\delta_1= \eta_1$. For $n \ge 2$, since $a_{n-1}
\precsim a_{n}$, there is $\delta_{n} \in (0,\eta_{n}]$ such that
$(a_{n-1}-\delta_{n-1})_+ \precsim (a_{n}-\delta_{n})_+$. To see
that the sequence $\{(a_n-\delta_n)_+\}_{n=1}^\infty$ is full in
$A$, let $I$ be the closed two-sided ideal generated by this
sequence. Since $(a_j-1/n)_+ \precsim (a_n-\eta_n)_+ \precsim
(a_n-\delta_n)_+ \in I$ whenever $1 \le j < n$, we see that
$(a_j-1/n)_+$ belongs to $I$ whenever $n > j$. It follows that $a_j$
belongs to $I$ for all $j$, whence $I=A$, because the sequence
$\{a_n\}$ was assumed to be full.)

Next we choose $\delta'_n >0$ such that $(a_n-\delta_n)_+ \precsim_m
(b_n-\delta'_n)_+$ for all $n$. Let $g_n \colon [0,1] \to [0,1]$ be
the continuous function given by $g_n(0)=0$, $g_n(t)=1$ for $t \ge
\delta'_n$, and $g_n$ is linear on $[0,\delta'_n]$. Put $b'_n =
g_n(b_n)$. Then $b_n$ is Cuntz equivalent to $b_n'$, and
$(b_n-\delta'_n)_+ \precsim (b'_n-1/2)_+$.

We claim that $T:=\bigoplus_{n=1}^\infty b'_n$ is full
  in $\mathcal{M}(A)$. To this end, take a strictly positive
  contraction $c$ in $A$.
Let $k \in \N$ and $\ep>0$ be given. Note that the tail
$\{(a_n-\delta_n)_+\}_{n=k}^\infty$ is full in $A$ (because the
sequence $\{(a_n-\delta_n)_+\}_{n=1}^\infty$ is Cuntz increasing).
It follows that $c$ belongs to the closed two-sided ideal generated
by $\{(a_n-\delta_n)_+\}_{n=k}^\infty$, whence $(c-\ep)_+$ belongs
to the algebraic ideal generated by this sequence, and hence to the
algebraic ideal generated by $\{(a_n-\delta_n)_+\}_{n=k}^{k'}$ for
some $k'>k$. This entails that
$$(c-\ep)_+ \precsim_p (a_k-\delta_k)_+ \oplus
(a_{k+1}-\delta_{k+1})_+ \oplus \cdots \oplus (a_{k'}-\delta_{k'})_+
,$$ for some positive integer $p$. Using again the sequence
$\{(a_n-\delta_n)_+\}_{n=1}^\infty$ is Cuntz increasing, we get that
\begin{eqnarray*}
(c-\ep)_+ & \precsim & (a_k-\delta_k)_+ \oplus
(a_{k+1}-\delta_{k+1})_+ \oplus \cdots \oplus
(a_{\ell}-\delta_{\ell})_+ \\
& \precsim_m & (b'_k-1/2)_+ \oplus (b'_{k+1}-1/2)_+ \oplus \cdots
\oplus (b'_\ell - 1/2)_+,
\end{eqnarray*}
when $\ell \ge k + p(k'-k+1)$. Lemma~\ref{lm:2} now yields that $T$
is full in $\mathcal{M}(A)$.

Since $A$ is assumed to have the Corona Factorization Property we
can use Lemma~\ref{lm:1} to conclude that $a_1 \precsim T$. Hence
$(a_1-\eta/2)_+ = R^*TR$ for some $R$ in $\mathcal{M}(A)$. Take
a positive contraction $e$ in $A$ such that
$e(a_1-\eta/2)_+=(a_1-\eta/2)_+=(a_1-\eta/2)_+e$. Put $r= Re \in A$.
As $\bigoplus_{n=1}^k b'_n \to T$ in the strict topology as $k \to
\infty$, it follows that $r^*\big(\bigoplus_{n=1}^k b'_n\big)r \to
r^*Tr = (a_1-\eta/2)_+$ in the norm topology (on $A$) as $k \to \infty$.
Take $k$ such that
$$\big\|r^* \big(b'_1 \oplus b'_2 \oplus \cdots \oplus b'_k \big)r -
(a_1-\eta/2)_+\big\| < \eta/2.$$ Then
$$(a_1-\eta)_+ \, \precsim \, r^*  \big(b'_1 \oplus b'_2 \oplus \cdots
\oplus b'_k \big)r \, \precsim \, b'_1 \oplus b'_2 \oplus \cdots
\oplus b'_k \, \approx \, b_1 \oplus b_2 \oplus \cdots \oplus b_k$$
as desired.
\end{proof}

\begin{theor}\label{cuntz_CFP} Let $A$ be a $\sigma$-unital \Cs. Then
  $A$ has the
  Corona Factorization Property if and only if its Cuntz semigroup,
  $\Cu(A)$, has the Corona Factorization Property (for semigroups).
\end{theor}

\begin{proof} Assume first that $A$ has the Corona Factorization
  Property. 
Let $\{x_n\}$ be a full sequence in $\Cu(A)$, let $\{y_n\}$
  be another sequence in $\Cu(A)$, let $x' \in \Cu(A)$, and let $m \in \N$
  be such that $x_n \le my_n$ for all $n$ and $x' \ll x_1$. Take
  positive elements $a_n$ and $b_n$ in $A\otimes\cK$ such that $x_n = \langle
  a_n \rangle$ and $y_n = \langle b_n \rangle$, and take $\eta > 0$
  such that $x' \le \langle (a_1-\eta)_+ \rangle$. Then $\{a_n\}$ is
  full in $A\otimes\cK$, $a_1 \precsim a_2 \precsim \cdots$, and $a_n
  \precsim
  b_n \otimes 1_m$ for all $n$. Hence, by Lemma~\ref{lm:3}, we get that
$$(a_1-\eta)_+ \precsim b_1 \oplus b_2 \oplus \cdots \oplus b_k$$
for some $k$. Thus
$$x' \le \langle (a_1-\eta)_+ \rangle \le \langle b_1 \rangle +
\langle b_2 \rangle + \cdots  +\langle b_k \rangle = y_1 + y_2 +
\cdots + y_k.$$ This shows that $\Cu(A)$ has the Corona Factorization
Property.

For the converse direction it is shown in \cite[Theorem 4.2]{KN} that
$A$ has the Corona Factorization Property if and only if for every
full hereditary sub-\Cs{} $B$ of $A$ one has that $B$ is stable if
some matrix algebra over $B$ is stable. As $\Cu(B) \cong \Cu(A)$ for
all such $B$,  this follows from
Proposition~\ref{prop:CFP_stable}.
\end{proof}

\begin{corol}
Let $A$ be a separable, unital $C^*$-algebra with finite decomposition
rank. Then $A$ has the Corona Factorization Property.
\end{corol}

\begin{proof} Combine Theorem~\ref{cuntz_CFP} above with
  Corollary~\ref{cor:fd}.
\end{proof}

\noindent The corollary above extends the result of  Pimsner, Popa
and Voiculescu, \cite{PPV}, and Kucerovsky and Ng, \cite{KN2}, that
the  $C^*$-algebra $C(X)\otimes \mathcal{K}$ is absorbing, or
equivalently, that it satisfies the Corona Factorization Property,
when $X$ has finite covering dimension (as the decomposition rank of
$C(X)\otimes \mathcal{K}$ coincides with the covering dimension of
the space $X$).

We end this paper by describing for which \Cs s the Cuntz semigroup
has the strong Corona Factorization Property.

\begin{theor} \label{sCFP-C*-alg}
Let $A$ be a separable \Cs. Then $\Cu(A)$ has the
  strong Corona Factorization Property if and only if every ideal in
  $A$ has the Corona Factorization Property.
\end{theor}

\begin{proof} Assume that $\Cu(A)$ has the strong Corona Factorization
  Property, and let $I$ be a closed two-sided ideal in $A$. Then
  $\Cu(I)$ is an ideal in $\Cu(A)$. As the strong Corona Factorization
  Property trivially passes to ideals, we conclude that $\Cu(I)$
  satisfies the (strong) Corona
  Factorization Property. It therefore follows from
  Theorem~\ref{cuntz_CFP} that $I$ has the Corona Factorization Property (for
  \Cs s).

Suppose now that all ideals in $A$ have the Corona Factorization
Property. To show that $\Cu(A)$ has the strong
Corona Factorization Property, it suffices to show that whenever
$a,b_1,b_2, \dots$ are positive elements in $A\otimes \cK$, $\ep >0$,
and $m$ is a positive integer such that $a \precsim b_n \otimes 1_m$,
then $(a-\ep)_+ \precsim b_1 \oplus b_2 \oplus \cdots \oplus b_k$ for
some positive integer $k$.  There are
elements $t_n \in A\otimes\cK$ such that $t_n^*(b_n
\otimes 1_m)t_n=(a-\ep/3n)_+$.

Let $I$ be the closed two-sided ideal in $A\otimes \cK$ generated by $a$. Then
$a$ is full in $I$, and $\langle a \rangle$ is full in $\Cu(I)$; however,
the elements $b_n$ may not belong to $I$. To fix this problem, take a quasi-central
increasing approximate unit $\{e_k\}_{k=1}^\infty$ for $I$ consisting of
positive contractions. For each $n$ find $k$ such that
$$\|t_n^*(e_k \otimes 1_{m}) (b_n \otimes 1_m) (e_k \otimes
1_{m})t_n  - (a-\ep/3n)_+\| < \ep/3n,$$
and put $a_n = (a-2\ep/3n)_+$ and
$c_n = e_k b_n e_k$.
Then $c_n$ belongs to $I$, $c_n \precsim b_n$,
$a_n \precsim c_n \otimes
1_m$ (relatively to $A$, and hence also relatively to $I$), $(a-\ep)_+ =
(a_1-\ep/3)_+$,
and $\{\langle a_n \rangle \}$ is a full sequence in
$\Cu(I)$. Since $\Cu(I)$ is assumed to satisfy the Corona Factorization
Property we conclude that
$$(a-\ep)_+ \precsim c_1 \oplus c_2 \oplus \cdots \oplus c_k \precsim
b_1 \oplus b_2 \oplus \cdots \oplus b_k$$
for some $k$ as desired.
\end{proof}

\section*{Acknowledgements}

The first and second named authors were partially supported by a
MEC-DGESIC grant (Spain) through Project MTM2008-0621-C02-01/MTM,
and by the Comissionat per Universitats i Recerca de la Generalitat
de Catalunya. The third named author was supported by a grant from
the Danish Natural Science Research Council (FNU). Part of this
research was carried out during visits of the first and third named
authors to UAB (Barcelona), of the second named author to SDU
(Odense), and of the two first mentioned authors to Copenhagen. We
wish to thank all parties involved for the hospitality extended to
us.


\begin{thebibliography}{99}

\bibitem{ABP} {\sc Antoine, R., Bosa, J. and Perera, F.}, Completions
  of monoids with applications to the Cuntz semigroup. Preprint.

\bibitem{BH} {\sc Blackadar, B. and  Handelman, D.}, Dimension
  functions and traces on $C\sp{\ast} $-algebras,  \emph{J. Funct. Anal.}
  \textbf{45}  (1982), no. 3, 297--340.

\bibitem{BoRo} {\sc Bosa, J. and  R\o rdam, M.}, In preparation.


\bibitem{CEI} {\sc Coward, K.T., Elliott, G.A., Ivanescu, C.}, The
  Cuntz semigroup as an invariant for $C\sp *$-algebras,  \emph{J.
Reine Angew. Math.}  \textbf{623}  (2008), 161--193.

\bibitem{Cun} {\sc Cuntz, J.}, Dimension Functions on Simple
$C^*$-algebras, \emph{Math. Ann.} \textbf{233} (1978), 145--153.

\bibitem{EK} {\sc Elliott, G.A. and Kucerovsky, D.}, An abstract
  Voiculescu-Brown-Douglas-Fillmore absorption theorem.  \emph{Pacific
    J. Math.}  \textbf{198}  (2001),  no. 2, 385--409.

\bibitem{GH:Extending} {\sc Goodearl, K. R. and Handelman, D.}, Rank
  functions and $K_0$ of regular rings, \emph{J. Pure Appl. Algebra},
  \textbf{7} (1976), 195--216.

\bibitem{HRW} {\sc  Hirshberg, I.,  R\o rdam, M. and  Winter, W.},
  $C\sb 0(X)$-algebras, stability and strongly self-absorbing $C\sp
  *$-algebras,  \emph{Math. Ann.}  \textbf{339}  (2007),  no. 3,
  695--732.


\bibitem{HR} {\sc  Hjelmborg, J. and  R\o rdam, M.}, On stability of
  $C\sp *$-algebras,  \emph{J. Funct. Anal.}  \textbf{155}  (1998),  no.
1, 153--170.

\bibitem{KR} {\sc  Kirchberg, E. and R\o rdam, M.}, Non-simple purely
  infinite $C\sp \ast$-algebras,  \emph{Amer. J. Math.}  \textbf{122}
  (2000),  no. 3, 637--666.

\bibitem{KW} {\sc  Kirchberg, E. and Winter, W.}, Covering dimension
  and quasidiagonality, \emph{International J. Math.} \textbf{15} (2004), 63--85.


\bibitem{KN} {\sc  Kucerovsky, D. and  Ng, P. W.}, $S$-regularity and
  the corona factorization property,  \emph{Math. Scand.}  \textbf{99}
  (2006),  no. 2, 204--216.

\bibitem{KN2} {\sc  Kucerovsky, D. and  Ng, P. W.}, Decomposition rank
  and absorbing extensions of type I algebras,  \emph{J. Funct. Anal.}
  \textbf{221}  (2005),  no. 1, 25--36.

\bibitem{OPR} {\sc Ortega, E., Perera, F. and R\o rdam, M.}, The
  Corona Factorization Property and Refinement Monoids. Preprint, 2009.


\bibitem{PPV} {\sc Pimsner, M., Popa, S. and Voiculescu, D.},
  Homogeneous $C\sp{\ast} $-extensions of $C(X)\otimes K(H)$. I,
  \emph{J. Operator Theory} \textbf{1}  (1979), no. 1, 55--108.

\bibitem{Robert} {\sc Robert, L.}, Nuclear dimension and
  $n$-comparison. Preprint, 2009. 

\bibitem{Ror:JOT} {\sc R{\o}rdam, M.}, Ideals in the Multiplier
  Algebra of a Stable $C\sp \ast$-algebra,
  \emph{J. Operator Theory} \textbf{25} (1991), no. 2, 283--298.

\bibitem{Ro:UHFII} {\sc R{\o}rdam, M.}, On the Structure of Simple
  $C^*$-algebras Tensored with a UHF-Algebra, II,
  \emph{J. Funct. Anal.} \textbf{107} (1992), 255--269.

\bibitem{Ro2} {\sc  R\o rdam, M.}, Stability of $C\sp *$-algebras is
  not a stable property,  \emph{Doc. Math.}  \textbf{2}  (1997), 375--386

\bibitem{Ro4} {\sc  R\o rdam, M.}, On sums of finite projections, in
  ``Operator algebras and operator theory (Shanghai, 1997)'' Amer.\
Math.\ Soc., Providence, RI. (1998), 327--340.


\bibitem{Ro3} {\sc  R\o rdam, M.}, A simple $C\sp *$-algebra with a
  finite and an infinite projection,  \emph{Acta Math.}  \textbf{191}
  (2003),  no. 1, 109--142.


\bibitem{Ro1} {\sc  R\o rdam, M.}, The Stable and the Real Rank of
  $\mathcal{Z}$-absorbing $C^*$-algebras,  \emph{International J. Math.}
  \textbf{15}  (2004),  no. 10, 1065--1084.

\bibitem{TW} {\sc  Toms, A. and Winter, W.}, The Elliott Conjecture
  for Villadsen Algebras of first type, Preprint.

% \bibitem{W} {\sc  Winter, W.}, On Topological finite-dimensional simple
% $C^*$-algebras, \emph{Math. Ann.} \textbf{332} (2005) 843--878.

% \bibitem{Wi1} {\sc  Winter, W.}, Decomposition rank and
%   $\mathcal{Z}$-stability, Preprint, 2008.
\end{thebibliography}
\end{document}